\documentclass{article}
\usepackage[margin=1in]{geometry}

\usepackage[nospace,noadjust]{cite}
\usepackage{amsmath,amssymb,amsfonts,times}
\usepackage{graphicx}
\usepackage{textcomp}
\usepackage{xcolor}
\usepackage{subfigure}
\usepackage{epsfig}
\usepackage{multirow}
\usepackage{algorithm}
\usepackage{algorithmicx}
\usepackage{algpseudocode}
\usepackage{array}
\usepackage{epstopdf}
\usepackage{color}
\usepackage{diagbox}
\usepackage{bm}

\usepackage{pdfsync}

\usepackage{url}
\usepackage[hidelinks]{hyperref}
\usepackage{amsmath,amsthm,amssymb,amsbsy}
\usepackage{paralist}
\usepackage{xcolor}
\usepackage{color}
\usepackage{graphicx}
\graphicspath{{./figs/}}
\usepackage{comment}
\usepackage{multirow}
\usepackage[toc, page]{appendix}

\usepackage{graphicx}
\usepackage{fancyhdr}
\usepackage{cite}
\usepackage{cleveref}

\newtheorem{lemma}{Lemma}

\newtheorem{theorem}{Theorem}

\theoremstyle{remark}

\theoremstyle{problem}

\renewcommand{\O}{\mathbb{O}}
\newcommand{\R}{\mathbb{R}}
\def \real    { \mathbb{R} }

\newcommand{\e}{\begin{equation}}
\newcommand{\ee}{\end{equation}}
\newcommand{\en}{\begin{equation*}}
\newcommand{\een}{\end{equation*}}
\newcommand{\eqn}{\begin{eqnarray}}
\newcommand{\eeqn}{\end{eqnarray}}
\newcommand{\bmat}{\begin{bmatrix}}
\newcommand{\emat}{\end{bmatrix}}

\DeclareMathAlphabet\mathbfcal{OMS}{cmsy}{b}{n}
\renewcommand{\P}[1]{\operatorname{\mathbb{P}}\left(#1\right)}
\newcommand{\E}{\operatorname{\mathbb{E}}}

\newcommand{\vct}[1]{\boldsymbol{#1}}
\newcommand{\mtx}[1]{\boldsymbol{#1}}

\newcommand{\<}{\langle}
\renewcommand{\>}{\rangle}

\newcommand{\rank}{\operatorname{rank}}

\newcommand{\dist}{\operatorname{dist}}

\newcommand{\set}[1]{\mathbb{#1}}

\DeclareMathOperator*{\argmin}{\text{arg~min}}

\def \st {\operatorname*{s.t.\ }}

\newcommand{\ol}{\overline}

\newcommand{\calA}{\mathcal{A}}
\newcommand{\calB}{\mathcal{B}}
\newcommand{\calC}{\mathcal{C}}
\newcommand{\calD}{\mathcal{D}}
\newcommand{\calE}{\mathcal{E}}

\newcommand{\calH}{\mathcal{H}}
\newcommand{\calI}{\mathcal{I}}

\newcommand{\calN}{\mathcal{N}}

\newcommand{\calP}{\mathcal{P}}

\newcommand{\calS}{\mathcal{S}}

\newcommand{\calX}{\mathcal{X}}
\newcommand{\calY}{\mathcal{Y}}

\newcommand{\vh}{\vct{h}}

\newcommand{\vr}{\vct{r}}
\newcommand{\vs}{\vct{s}}

\newcommand{\vx}{\vct{x}}
\newcommand{\vy}{\vct{y}}

\newcommand{\mA}{\mtx{A}}
\newcommand{\mB}{\mtx{B}}

\newcommand{\mD}{\mtx{D}}
\newcommand{\mE}{\mtx{E}}

\newcommand{\mG}{\mtx{G}}

\newcommand{\mR}{\mtx{R}}

\newcommand{\mU}{\mtx{U}}

\newcommand{\mX}{\mtx{X}}
\newcommand{\mY}{\mtx{Y}}

\newcommand{\mId}{{\bf I}}

\newcommand{\setI}{\set{I}}

\graphicspath{{./figs/}}

\newlength{\imgwidth}
\newboolean{twoColVersion}
\setboolean{twoColVersion}{false}
\newcommand{\twoCol}[2]{\ifthenelse{\boolean{twoColVersion}} {#1} {#2} }

\title{\LARGE \bf Robust Low-rank Tensor Train Recovery}

\author{Zhen Qin and Zhihui Zhu\thanks{ZQ (email: qin.660@osu.edu) and ZZ (email: zhu.3440@osu.edu)  are with the Department of Computer Science and Engineering, Ohio State University. }
}

\begin{document}

\maketitle

\begin{abstract}
Tensor train (TT) decomposition represents an $N$-order tensor using $O(N)$ matrices (i.e., factors) of small dimensions, achieved through products among these factors. Due to its compact representation, TT decomposition has found wide applications, including various tensor recovery problems in signal processing and quantum information.
In this paper, we study the problem of reconstructing a TT format tensor from measurements that are contaminated by outliers with arbitrary values. Given the vulnerability of smooth formulations to corruptions, we use an $\ell_1$ loss function to enhance robustness against outliers. We first establish the $\ell_1/\ell_2$-restricted isometry property (RIP) for Gaussian measurement operators, demonstrating that the information in the TT format tensor can be preserved using a number of measurements that grows linearly with $N$. We also prove the sharpness property for the $\ell_1$ loss function optimized over TT format tensors. Building on the $\ell_1/\ell_2$-RIP and sharpness property, we then propose two complementary methods to recover the TT format tensor from the corrupted measurements: the projected subgradient method (PSubGM), which optimizes over the entire tensor, and the factorized Riemannian subgradient method (FRSubGM), which optimizes directly over the factors. Compared to PSubGM, the factorized approach FRSubGM significantly reduces the memory cost at the expense of a slightly slower convergence rate. Nevertheless, we show that both methods, with diminishing step sizes, converge linearly to the ground-truth tensor given an appropriate initialization, which can be obtained by a truncated spectral method. To the best of our knowledge, this is the first work to provide a theoretical analysis of the robust TT recovery problem and to demonstrate that TT-format tensors can be robustly recovered even when up to half of the measurements are arbitrarily corrupted. We conduct various numerical experiments to demonstrate the effectiveness of the two methods in robust TT recovery.

\end{abstract}

\section{Introduction}
\label{intro}

Tensor recovery has been widely investigated in many areas, such as signal processing and machine learning \cite{CichockiMagTensor15,SidiropoulosTSPTENSOR17}, communication \cite{SidiropoulosBlind}, quantum physics \cite{francca2021fast, lidiak2022quantum,qin2024quantum}, chemometrics \cite{Smilde04,AcarUnsup09}, genetic engineering \cite{HoreNature16}, and so on.  One fundamental task is to recover a tensor $\calX^\star\in\R^{d_1\times\cdots\times d_N}$ from highly incomplete, sometimes even corrupted, observations $\vy = \{y_k\}_{k=1}^m$ given by
\begin{eqnarray}
    \label{Noisy Definition of tensor sensing L1 loss}
    \vy = \calA(\calX^\star) + \vs = \begin{bmatrix}
          y_1 \\
          \vdots \\
          y_m
        \end{bmatrix} = \begin{bmatrix}
          \<\calA_1,\calX^\star\> +s_1 \\
          \vdots \\
          \<\calA_m,\calX^\star\> +s_m
        \end{bmatrix} \in\R^m,
\end{eqnarray}
where $\calA(\calX^\star): \R^{d_{1}\times  \cdots \times d_{N}}\rightarrow \R^m$ is a linear observation operator that models the measurement process and  $\vs\in\R^m$ represents an outlier vector, wherein only a small fraction of its entries (referred to as outliers) have arbitrary magnitudes but their locations are unknown a prior, while the remaining entries are zero. In practical scenarios, outliers are frequently encountered in sensing or regression models \cite{de2003framework, li2004statistical, guo2011tensor, zhou2013tensor, li2016low, li2017parsimonious, hao2020sparse, tong2022accelerating, lidiak2022quantum, qin2024quantum}, stemming from various factors such as sensor malfunctions and malicious attacks. For instance, in quantum state tomography, imperfections during quantum state preparation can randomly generate unwarranted outlier quantum states, which subsequently lead to outliers during the measurement operation \cite{li2014robust, hara2014anomaly,li2016low}.

Even in the absence of outliers, the recovery from \eqref{Noisy Definition of tensor sensing L1 loss} remains ill-posed due to the curse of dimensionality, which arises from the exponential storage complexity of $\calX^\star$ with respect to $N$.
Therefore, it is often advantageous (and even necessary) to employ certain tensor decomposition models to compactly represent the full tensor.  One commonly used model is the tensor train (TT) decomposition \cite{Oseledets11}, which expresses the $(s_1,\dots,s_N)$-th element of $\calX^\star$ as the following matrix product form \cite{Oseledets11}
\begin{eqnarray}
    \label{Definition of Tensor Train in intro}
    \calX^\star(s_1,\dots,s_N)= \mX_1^\star(:,s_1,:) \mX_2^\star(:,s_2,:) \cdots \mX_N^\star(:,s_N,:),
\end{eqnarray}
where tensor factors ${\mX}_i^\star \in\R^{r_{i-1}\times d_i \times r_i}, i\in [N]$ with $r_0=r_N=1$. The dimensions $\text{rank}(\calX^\star) = (r_1,\dots, r_{N-1})$ of such a decomposition are called the TT ranks\footnote{Any tensor can be decomposed in the TT format \eqref{Definition of Tensor Train in intro} with sufficiently large TT ranks \cite[Theorem 2.1]{Oseledets11}. Indeed, there always exists a TT decomposition with $r_i \le\min\{\Pi_{j=1}^{i}d_j, \Pi_{j=i+1}^{N}d_j\}$ for any $i\ge 1$.} of $\calX^\star$. We say a TT format tensor is low-rank if $r_i$ is much smaller compared to $\min\{\Pi_{j=1}^{i}d_j, \Pi_{j=i+1}^{N}d_j\}$ for most indices $i$ so that the total number of parameters in the tensor factors $\{\mX_i^\star\}$ is much smaller than the number of entries in $\calX^\star$. We refer to any tensor for which such a low-rank TT decomposition exists as a {\it low-TT-rank} tensor. To simplify the notation, we may also use $[\mX_1^\star,\dots, \mX_N^\star]$ as the compact form of $\calX^\star$.

Compared to the other two commonly used tensor decompositions---canonical polyadic (CP) \cite{Bro97} and Tucker \cite{Tucker66} decompositions---TT decomposition strikes a balance between the advantages of both approaches\footnote{In general, finding the optimal CP decomposition for high-order tensors can be computationally difficulty~\cite{johan1990tensor,de2008tensor}, while the Tucker decomposition becomes inapplicable for high-order tensors due to the number of parameters scaling exponentially with the tensor order.}. The number of parameters of TT decomposition is $O(N\ol d\ol r^2)$ with $\ol d = \max_{i}d_i$ and $\ol r = \max_{i} r_i$, not growing exponentially with the tensor order as the CP decomposition. Furthermore, similar to the Tucker decomposition, the TT decomposition can be approximately computed using an SVD-based algorithm, called the tensor train SVD (TT-SVD), with a guaranteed accuracy \cite{Oseledets11}. See \cite{cichocki2016tensor} for a detailed description. Consequently, TT decomposition has been widely applied to tensor recovery across various fields, including quantum tomography \cite{qin2024quantum}, neuroimaging \cite{zhou2013tensor}, facial model refinement \cite{cai2024robust}, and the distinction of its attributes \cite{llosa2022reduced},  longitudinal relational data analysis \cite{hoff2015multilinear}, and forecasting tasks \cite{liu2020low}.

\paragraph{Our goals and main results} In this paper, we study the robust recovery problem in \eqref{Noisy Definition of tensor sensing L1 loss}, where the underlying tensor $\calX^\star$ has low TT ranks. We refer to this as the robust TT recovery problem. To handle outliers in the measurements, we employ a robust $\ell_1$ loss function together with the TT format and solve the following problem:
\begin{eqnarray}
    \label{Loss Function of robust tensor sensing abstract}
    \begin{split}
    \min_{\mbox{\tiny$\begin{array}{c}
     \calX\in\R^{d_1\times\cdots\times d_N}\\
     \text{rank}(\calX) = (r_1,\dots, r_{N-1})\end{array}$}} f(\calX) = \frac{1}{m}\|\calA(\calX)-\vy\|_1.
    \end{split}
\end{eqnarray}
Compared to the conventional least-squares loss, the $\ell_1$ loss function is more robust against outliers and has been widely adopted in robust signal recovery problems   \cite{candes2011robust, li2016low,josz2018theory, duchi2019solving, charisopoulos2019composite, li2020nonconvex, charisopoulos2021low, ma2021implicit, ding2021rank,tong2022accelerating}. However, the combination of the $\ell_1$ loss function and TT decomposition makes the problem \eqref{Loss Function of robust tensor sensing abstract} highly nonsmooth and nonconvex. Our goal is to study its optimality conditions and develop optimization algorithms with guaranteed performance.

Note that measurements should satisfy certain properties to enable robust recovery from corrupted measurements. Thus, we first study the stable embedding of low-TT-rank tensors by establishing the following $\ell_1/\ell_2$-restricted isometry property ($\ell_1/\ell_2$-RIP\footnote{$\ell_1/\ell_2$-RIP differs from the $\ell_2/\ell_2$-RIP \cite{Rauhut17, qin2024quantum} which examines the relationship between $\|\calA(\calX)\|_2^2$ and $\|\calX\|_F^2$.}) without outliers for $\calA$, which has been introduced previously in the context of low-rank matrix/Tucker tensor recovery \cite{zhang2013restricted,yue2016perturbation,li2020nonconvex,tong2022accelerating,tong2022scaled} and covariance estimation \cite{chen2015exact}. This mixed-norm approximate isometry evaluates the signal strength before and after projection using different metrics: the input is measured in terms of the Frobenius norm, and the output is measured in terms of the $\ell_1$ norm. Specifically, we say $\calA$ satisfies rank-$\ol r$ $\ell_1/\ell_2$-RIP if there exits a constant $\delta_{\ol r} \in (0,\sqrt{2/\pi})$ such that
\begin{eqnarray}
    \label{L1 RIP inequality intro}
    (\sqrt{2/\pi}-\delta_{\overline{r}})\|\calX\|_F \leq \frac{1}{m}\|\calA(\calX)\|_1 \leq (\sqrt{2/\pi}+\delta_{\overline{r}})\|\calX\|_F
\end{eqnarray}
holds for all low-TT-rank tensors with ranks $(r_1,\dots, r_{N-1}), r_i \le \ol r$. We show that Gaussian measurement operators $\calA$, where $\calA_1,\dots, \calA_m$ have independent and identically distributed (i.i.d.) standard Gaussian entries, satisfies $\ell_1/\ell_2$-RIP \eqref{L1 RIP inequality intro} with high probability as long as $m\geq \Omega(N\overline{d}\overline{r}^2\log N/\delta_{\overline{r}}^2)$ with $\ol d = \max_i d_i$. This implies that robust TT recovery is possible using a number of measurements that only scale (approximately) linearly with regard to $N$. With the $\ell_1/\ell_2$-RIP property, we show that the robust loss function in \eqref{Loss Function of robust tensor sensing abstract} satisfies the sharpness property \cite{burke1993weak,marcotte1998weak,li2020nonconvex, tong2021low, tong2022accelerating,tong2022scaled}: for any low-TT-rank tensors $\calX$ with TT ranks $r_i \le \ol r$, it holds that
\begin{eqnarray}
    \label{L1 RIP property intro sharpness}
    \frac{1}{m}\|\calA(\calX - \calX^\star) -\vs \|_1 - \frac{1}{m}\|\vs\|_1 \geq ((1-2p_s)\sqrt{2/\pi} - \delta_{\overline{r}})\|\calX - \calX^\star\|_F,
\end{eqnarray}
where $p_s \in[0,0.5]$ represents the fraction of outliers in $\vy$, i.e., $p_s = \|\vs\|_0/m$.  Since \eqref{Loss Function of robust tensor sensing abstract} optimizes only over low-TT-rank tensors, \eqref{L1 RIP property intro sharpness} needs to hold only for these tensors; as such, a similar condition for Tucker tensors is also referred to as {\it restricted sharpness} in \cite{tong2022accelerating,tong2022scaled}. The sharpness condition \eqref{L1 RIP property intro sharpness} implies that $\calX^\star$ is the unique global minimum, with the loss function increasing as the variable $\calX$ deviates from $\calX^\star$.

Our second contribution is to propose two complementary iterative algorithms for solving \eqref{Loss Function of robust tensor sensing abstract}. Building on insights from \cite{Rauhut17}, we introduce a projected subgradient method (PSubGM). This method optimizes the entire tensor in each iteration and employs the TT-SVD to project the iterates back to the TT format. Under the sharpness property, we establish a robust regularity condition (RRC) for the objective function \eqref{Loss Function of robust tensor sensing abstract}. We show that the PSubGM algorithm, with appropriate initialization and diminishing step sizes, achieves a linear convergence rate. Remarkably, PSubGM can precisely recover the ground-truth tensor $\calX^\star$ even in the presence of outliers.

A potential drawback of PSubGM when handling high-order tensors is that it requires storing the full estimated tensor $\calX$ and performing TT-SVD at each iteration, which becomes impractical for large $N$, such as in quantum state tomography involving hundreds of qubits \cite{qin2024quantum}.
To address this issue, instead of optimizing directly over the tensor $\calX$, we employ the factorization approach that optimizes over the factors $\{\mX_i\}_{i\geq 1}$ which can significantly reduce the memory cost. Specifically, we consider the following optimization problem:
\begin{eqnarray}
    \label{Loss Function of robust tensor sensing tensor factor  abstract}
    \begin{split}
    \min_{\mbox{\tiny$\begin{array}{c}
     \mX_i\in\R^{r_{i-1}\times d_i\times r_i},\\
     i\in[N]\end{array}$}} &  \frac{1}{m}\|\calA([\mX_1,\dots,\mX_N])-\vy\|_1,\\
     &\st \ \sum_{s_i=1}^{d_i} \mX_i^\top(:,s_i,:)\mX_i(:,s_i,:) =\mId_{r_i},  \ \ i\in[N-1].
    \end{split}
\end{eqnarray}
The additional constraints $\sum_{s_i=1}^{d_i} \mX_i^\top(:,s_i,:)\mX_i(:,s_i,:) =\mId_{r_i}$ are introduced to reduce the scaling ambiguity of the factors \cite{qin2024guaranteed}. The orthogonality constraints can be viewed as Stiefel manifolds of Riemannian space, so we utilize a factorized Riemannian subgradient method (FRSubGM) on the Stiefel manifold to optimize \eqref{Loss Function of robust tensor sensing tensor factor  abstract}.
We show that the objective function \eqref{Loss Function of robust tensor sensing tensor factor abstract} also satisfies a Riemannian RRC, and prove that the FRSubGM algorithm, with an appropriate initialization and a diminishing step size, converges to the ground-truth tensor $\calX^\star$ at a linear rate.  Finally, we present a guaranteed truncated spectral initialization as a valid starting point, ensuring linear convergence for both the PSubGM and  FRSubGM algorithms.

In Table~\ref{Comparison among different tensor sensing}, we summarize the convergence results for PSubGM and FRSubGM and compare them with previous results on tensor recovery in the absence of outliers, specifically the IHT \cite{Rauhut17,qin2024computational} and factorized Riemannian gradient descent (FRGD) \cite{qin2024guaranteed} that solve problems similar to \eqref{Loss Function of robust tensor sensing abstract} and the factorized problem  \eqref{Loss Function of robust tensor sensing tensor factor  abstract}, with the objective function being changed to a smooth $\ell_2$ loss function. We observe that PSubGM and FRSubGM achieve a similar linear convergence rate as their smooth counterparts, demonstrating that the outliers can be handled as easily as in the noiseless case by the nonconvex optimization approaches. The convergence rate of IHT/PSubGM primarily hinges on the RIP constant, with a potential decay ($1 + c$) (where $c$ is a universal constant) owing to the expansiveness of the TT-SVD. Conversely, the convergence rate of FRGD/FRSubGM relies not only on the RIP constant but also on factors like $N$, $\ol r$, and $\calX^\star$, which could impede the convergence speed.

\begin{table}[!ht]
\renewcommand{\arraystretch}{1.4}
\begin{center}
\caption{Comparison of IHT/FRGD for solving TT recovery with noiseless measurements and PSubGM/FRSubGM for the corrupted measurements. Here $p_s$ denotes the fraction of outliers in the measurements. The upper bound of initialization is expressed in terms of $\|\calX^{(0)} - \calX^\star\|_F$. The convergence rates of IHT, PSubGM and FRGD, FRSubGM are respectively analyzed concerning $\|\calX^{(t)} - \calX^\star\|_F^2$ and $\text{dist}^2(\{\mX_i^{(t)}\},\{ \mX_i^\star\})$ defined in \eqref{BALANCED NEW DISTANCE BETWEEN TWO TENSORS}. $\delta_{\overline{r}}'\in(0,1)$ is a constant in standard $\ell_2/\ell_2$-RIP (see \cite[Theorem 2]{qin2024quantum}), while $\delta_{\overline{r}}\in(0,1)$ is a constant in $\ell_1/\ell_2$-RIP. $c$ is a universal constant.}
\label{Comparison among different tensor sensing}
{\begin{tabular}{|c||c|c|c|c|}\hline  {Algorithm} & {Outlier} &{Initialization Requirement} & {Rate of Convergence }& {RIP condition}
\\\hline\hline {IHT \cite{Rauhut17,qin2024computational}} & $\times$ &  $\frac{c\underline{\sigma}({\calX^\star})}{600N}$ & $(1+c)\big(1- \frac{(1 - \delta_{4\overline{r}}')^2}{2(1 + \delta_{4\overline{r}}')^2}\big)<1$ & $\delta_{4\overline{r}}' \leq \frac{1-\sqrt{2c/(1+c)}}{1+\sqrt{2c/(1+c)}}$
\\\hline {FRGD \cite{qin2024guaranteed}} & $\times$ &  $O\big(\frac{\underline{\sigma}^2(\calX^\star)}{\ol r N^2 \ol\sigma(\calX^\star)}\big)$ & $1 - O\big(\frac{(4-15\delta_{(N+3)\ol r}')^2}{(1 + \delta_{(N+3)\ol r}')^2N^2\ol r\kappa^2(\calX^\star)}\big)$  &  $\delta_{(N+3)\ol r}'\leq \frac{4}{15}$
\\\hline {PSubGM} & $\surd$ &  $\frac{c\underline{\sigma}({\calX^\star})}{600N}$ & $(1+c)\big(1-\frac{3((1-2p_s)\sqrt{2/\pi} - \delta_{2\overline{r}})^2}{4(\sqrt{2/\pi} + \delta_{2\overline{r}})^2}\big)<1$ & $\delta_{2\overline{r}}<\frac{1-2p_s-\sqrt{4c/(3+3c)}}{1+\sqrt{4c/(3+3c)}}\sqrt{\frac{2}{\pi}}$
\\\hline {FRSubGM} & $\surd$ &  $O\big(\frac{\underline{\sigma}^2(\calX^\star)}{\ol r N^2  \ol\sigma(\calX^\star)}\big)$ & $1- O\big(\frac{ ((1-2p_s)\sqrt{2/\pi} -  \delta_{(N+1)\overline{r}})^2 }{N^2 \ol r(\sqrt{2/\pi} + \delta_{(N+1)\overline{r}})^2\kappa^2(\calX^\star)}\big) $ & $\delta_{(N+1)\overline{r}}\leq(1-2p_s)\sqrt{\frac{2}{\pi}}$ \\\hline
\end{tabular}}{}
\end{center}
\end{table}

\paragraph{Related works}
Theoretical analyses and algorithmic designs for robust low-rank matrix recovery via nonsmooth optimization have been extensively studied in \cite{zhang2016provable,li2020nonconvex,ding2021rank,tong2021low,ma2023global}. A notable advantage of nonsmooth formulations is the enhanced robustness to adversarial outliers, achieved through a simple algorithmic design--the low-rank factors are updated in essentially the same manner, irrespective of the presence of outliers. However, existing theoretical frameworks for asymmetric matrix factorization cannot be extended to robust high-order tensor recovery, as the additional regularization terms introduced to balance the factors may not generalize to multiple tensor factors.

For tensor recovery from a limited number of measurements, most existing theoretical work and algorithmic designs have predominantly focused on developing optimization algorithms for either the noiseless case or the presence of Gaussian noise. Typically, a smooth loss function, such as the residual sum of squares ($\ell_2$ loss), is employed. Variants of projected gradient descent (PGD) algorithms, including iterative hard thresholding (IHT) \cite{Rauhut17, chen2019non, grotheer2022iterative} and Riemannian gradient descent on the fixed-rank manifold \cite{budzinskiy2021tensor, luo2022tensor}, have been studied for operating on the entire tensor with guaranteed convergence and performance. However, direct optimization over the tensor $\calX$ poses a challenge due to its exponentially large memory requirements in terms of $N$. To address this storage issue, factorization approaches \cite{TongTensor21, Han20, qin2024guaranteed} have been developed to optimize the factors of a tensor decomposition.

In contrast, tensor recovery from measurements corrupted by outliers has been less studied. Recently, the work \cite{tong2022accelerating} introduced the first provably scalable gradient descent algorithm for order-$3$ Tucker recovery from corrupted measurements. To the best of our knowledge, there is a lack of analysis and algorithmic design with guaranteed convergence for robust TT recovery.

{\bf Notation}: We use calligraphic letters (e.g., $\calY$) to denote tensors,  bold capital letters (e.g., $\mY$) to denote matrices, except for $\mX_i$ which denotes the $i$-th order-$3$ tensor factors in the TT format,  bold lowercase letters (e.g., $\vy$) to denote vectors, and italic letters (e.g., $y$) to denote scalar quantities.  Elements of matrices and tensors are denoted in parentheses, as in Matlab notation. For example, $\calX(i_1, i_2, i_3)$ denotes the element in position
$(i_1, i_2, i_3)$ of the order-3 tensor $\calX$.
The inner product of $\calA\in\R^{d_1\times\dots\times d_N}$ and $\calB\in\R^{d_1\times\dots\times d_N}$ can be denoted as $\<\calA, \calB \> = \sum_{s_1=1}^{d_1}\cdots \sum_{s_N=1}^{d_N} \calA(s_1,\dots,s_N)\calB(s_1,\dots,s_N) $.
The vectorization of  $\calX\in\R^{d_1\times\dots\times d_N}$, denoted as $\text{vec}(\calX)$, transforms the tensor $\calX$ into a vector. The $(s_1, \dots, s_N)$-th element of $\calX$ can be found in the vector $\text{vec}(\calX)$ at the position $s_1 + d_1(s_2-1) + \cdots + d_1d_2 \cdots d_{N-1}(s_N-1)$.
$\|\calX\|_F = \sqrt{\<\calX, \calX \>}$ is the Frobenius norm of $\calX$.
$\|\mX\|$ and $\|\mX\|_F$ respectively represent the spectral norm and Frobenius norm of $\mX$.
$\sigma_{i}(\mX)$ is the $i$-th singular value of $\mX$.
$\|\vx\|_2$ denotes the $\ell_2$ norm of $\vx$.
For a positive integer $K$, $[K]$ denotes the set $\{1,\dots, K \}$.
For two positive quantities $a,b\in \real$,  $b = O(a)$ means $b\leq c a$ for some universal constant $c$; likewise, $b = \Omega(a)$ represents $b\ge ca$ for some universal constant $c$.
To simplify notations in the following sections, for an order-$N$ TT format tensor with ranks $(r_1,\dots, r_{N-1})$, we define $\ol r=\max_{i=1}^{N-1} r_i$ and $\ol d=\max_{i=1}^N d_i$.

\section{$\ell_1/\ell_2$-Restricted Isometry Property and Sharpness for Robust TT Recovery}
\label{Basic Knowledge}

\subsection{Tensor train decomposition}
Recall the TT format in \eqref{Definition of Tensor Train in intro}. Considering that $\mX_{i}(:,s_i,:)$ will be extensively used, we  denote it by  $\mX_{i}(s_i)\in\R^{r_{i-1}\times r_{i}}$ as one ``slice'' of $\mX_i$ with the second index being fixed at $s_i$.
Thus, for any $\calX= [\mX_1,\dots,\mX_N ]\in\R^{d_1\times\dots\times d_N}$ in the TT format, we can express its $(s_1,\dots,s_N)$-th element as the following matrix product form
\begin{eqnarray}
    \label{Definition of Tensor Train}
    \calX(s_1,\dots,s_N)=\prod_{i=1}^N\mX_{i}(:,s_i,:) =\prod_{i=1}^N\mX_{i}(s_i).
\end{eqnarray}
We may also arrange the slices $\{ \mX_i(s_i)\}_{s_i=1}^{d_i}$ into the following form:
\begin{eqnarray}
    \label{left unfolding of tensor factor of Tensor Train}
    L(\mX_i)=\begin{bmatrix}\mX_i(1) \\ \vdots\\  \mX_i(d_i) \end{bmatrix}\in\R^{d_i r_{i-1}\times r_i}, \ \ \forall i\in[N],
\end{eqnarray}
where $L(\mX_i)$ is often referred to as the left unfolding of $\mX_i$ when viewing $\mX_i$ as a tensor.

The decomposition of the tensor $\calX$ into the form of \eqref{Definition of Tensor Train} is generally not unique: not only the factors $\mX_i(s_i)$  are not unique, but also the dimension of these factors can vary. According to \cite{holtz2012manifolds}, there exists a unique set of ranks $\vr= (r_1,\dots, r_{N-1})$ for which $\calX$ admits a minimal TT decomposition. We say the decomposition \eqref{Definition of Tensor Train} is minimal if the rank of the left unfolding matrix $L(\mX_i)$ in \eqref{left unfolding of tensor factor of Tensor Train} is $r_i$. In addition, the factors can be chosen such that $L(\mX_i)$ is orthonormal for all $i\in[N-1]$; that is
\begin{eqnarray}
    \label{left ofthgonal format of tensor factor of Tensor Train}
    L^\top(\mX_i)L(\mX_i) = \mId_{r_i}, \ \  i\in[N-1].
\end{eqnarray}
The resulting TT decomposition is called the left-orthogonal format of $\calX$. Moreover, in this case, $r_i$  equals to the rank of the $i$-th unfolding matrix $\calX^{\<i\>}\in\R^{(d_1\cdots d_k)\times (d_{k+1}\cdots d_N)}$ of the tensor $\calX$, where the $(s_1\cdots s_i, s_{i+1}\cdots s_N)$-th element \footnote{ Specifically, $s_1\cdots s_i$ and $s_{i+1}\cdots s_N$ respectively represent the $(s_1+d_1(s_2-1)+\cdots+d_1\cdots d_{i-1}(s_i-1))$-th row and $(s_{i+1}+d_{i+1}(s_{i+2}-1)+\cdots+d_{i+1}\cdots d_{N-1}(s_N-1))$-th column.
} of $\calX^{\<i\>}$ is given by $\calX^{\<i\>}(s_1\cdots s_i, s_{i+1}\cdots s_N) = \calX(s_1,\dots, s_N)$. This can also serve as an alternative way to define the TT rank.
With the $i$-th unfolding matrix $\calX^{\<i\>}$\footnote{We can also define the $i$-th unfolding matrix as $\calX^{\< i \>} = \mX^{\leq i}\mX^{\geq i+1}$, where each row of the left part $\mX^{\leq i}$ and each column of the right part $\mX^{\geq i+1}$ can be represented as $\mX^{\leq i}(s_1\cdots s_i,:) = \mX_1(s_1)\cdots \mX_i(s_i)$ and $\mX^{\geq i+1}(:,s_{i+1}\cdots s_{N}) = \mX_{i+1}(s_{i+1})\cdots \mX_N(s_{N})$. When factors are in left-orthogonal form, we have ${\mX^{\leq i}}^\top\mX^{\leq i} = \mId_{r_i}$ and $\sigma_j(\calX^{\< i\>}) = \sigma_j( \mX^{\geq i+1} ), j\in[N-1]$.} and \emph{TT ranks}, we can define its smallest singular value $\underline{\sigma}(\calX)=\min_{i=1}^{N-1}\sigma_{r_i}(\calX^{\<i\>})$, its largest singular value $\overline{\sigma}(\calX)=\max_{i=1}^{N-1}\sigma_{1}(\calX^{\<i\>})$ and condition number $\kappa(\calX)=\frac{\overline{\sigma}(\calX)}{\underline{\sigma}(\calX)}$.

\subsection{$\ell_1/\ell_2$-Restricted Isometry Property}

We first prove the $\ell_1/\ell_2$-RIP property for the robust TT recovery problem with Gaussian measurement operators, a “gold standard” for studying random linear measurements in the compressive sensing literature \cite{donoho2006compressed,candes2006robust,
candes2008introduction,recht2010guaranteed,CandsTIT11}. As previously studied in the contexts of low-rank matrix and Tucker tensor recovery problems \cite{zhang2013restricted, yue2016perturbation, li2020nonconvex, tong2022accelerating} and covariance estimation \cite{chen2015exact}, $\ell_1/\ell_2$-RIP establishes a connection between $\|\calA(\calX)\|_1$ and $\|\calX\|_F$, differing from previous work on $\ell_2/\ell_2$-RIP \cite{Rauhut17, qin2024quantum} on TT recovery problem, which examine the relationship between $\|\calA(\calX)\|_2^2$ and $\|\calX\|_F^2$.
\begin{theorem}
\label{L1 RIP}($\ell_1/\ell_2$-RIP of Gaussian measurement operators)
Suppose the linear map $\calA: \R^{d_{1}\times  \cdots \times d_{N}}\rightarrow \R^m$ is a Gaussian measurement operator where $\{\calA_k\}_{k=1}^m$ have i.i.d. standard Gaussian entries. Let $\delta_{\overline{r}}\in(0,\sqrt{2/\pi})$ be a positive constant. If the number of measurements satisfies $m\geq \Omega(N\overline{d}\overline{r}^2\log N/\delta_{\overline{r}}^2)$, then with probability exceeding $1-e^{-\Omega(N\overline{d}\overline{r}^2\log N)}$, $\calA$ satisfies the $\ell_1/\ell_2$-restricted isometry property in the sense that
\begin{eqnarray}
    \label{L1 RIP inequality}
    (\sqrt{2/\pi}-\delta_{\overline{r}})\|\calX\|_F \leq \frac{1}{m}\|\calA(\calX)\|_1 \leq (\sqrt{2/\pi}+\delta_{\overline{r}})\|\calX\|_F
\end{eqnarray}
hold for all low-TT-rank tensors $\calX$ with ranks $\vr = (r_1,\dots, r_{N-1})$.
\end{theorem}
The proof is provided in {Appendix} \ref{Proof of L1 RIP in appendix}. \Cref{L1 RIP} guarantees the RIP for Gaussian measurements where the number of measurements $m$ scales linearly, rather than exponentially, with respect to the tensor order $N$. When RIP holds, then for any two distinct TT format tensors $\calX_1,\calX_2$ with TT ranks smaller than $\ol r$, we have distinct measurements since
\begin{eqnarray}
    \label{L1 RIP inequality distinct two}
    \frac{1}{m}\|\calA(\calX_1) - \calA(\calX_2)\|_1 = \frac{1}{m}\|\calA(\calX_1 - \calX_2)\|_1 \geq (\sqrt{2/\pi}-\delta_{2\overline{r}})\|\calX_1 - \calX_2\|_F,
\end{eqnarray}
which guarantees the possibility of exact recovery in the absence of outliers. In addition, we note that \Cref{L1 RIP} can also be applicable to other measurement operators, such as subgaussian measurements \cite{vershynin2018high}, using a similar analysis.

\subsection{Sharpness}
\label{sec: sharpness}

We now study the $\ell_1$ loss function $f(\calX) = \frac{1}{m}\|\calA(\calX)-\vy\|_1$ and establish the sharpness property \cite{burke1993weak,marcotte1998weak} that can ensure exact recovery with corrupted measurements in \eqref{Noisy Definition of tensor sensing L1 loss}.
Let $\calS\subseteq \{1,\dots, m  \}$ denote the support of the outlier vector $\vs$, and $\calS^{c} = \{1,\dots, m   \} \backslash \calS$. We define $p_s = \frac{|\calS|}{m}$ as the fraction of outliers in $\vy$. The following result establishes the sharpness property for $\calA$.
\begin{lemma} (Sharpness of Gaussian measurement operators)
\label{sharpness property lemma}
Given an unknown target tensor $\calX^\star$ with ranks $\vr = (r_1,\dots, r_{N-1})$, suppose the linear map $\calA: \R^{d_{1}\times  \cdots \times d_{N}}\rightarrow \R^m$ is a Gaussian measurement operator. Let $\delta_{2\overline{r}}\in(0,(1-2p_s)\sqrt{2/\pi})$  be a positive constant. If the number of measurements satisfies $m\geq \Omega(N\overline{d}\overline{r}^2\log N/\delta_{2\overline{r}}^2)$, then with probability exceeding $1-2e^{-\Omega(N\overline{d}\overline{r}^2\log N)}$, $\calA$ satisfies the following sharpness property:
\begin{eqnarray}
    \label{L1 RIP property in lemma}
    \frac{1}{m}\|\calA(\calX - \calX^\star) -\vs \|_1 - \frac{1}{m}\|\vs\|_1 \geq ((1-2p_s)\sqrt{2/\pi} - \delta_{2\overline{r}})\|\calX - \calX^\star\|_F
\end{eqnarray}
holds for all low-TT-rank tensors $\calX$ with ranks $\vr$.
\end{lemma}
The proof is given in {Appendix} \ref{Proof of L1 RIP property in appendix}. To simplify the notation, we use $\delta_{2\overline{r}}$, as in \Cref{L1 RIP}, to represent the constant. \Cref{sharpness property lemma} establishes an exact recovery condition for measurements with outliers \eqref{Noisy Definition of tensor sensing L1 loss}, showing that when the outlier ratio $p_s\leq \frac{1}{2} - \frac{\delta_{2\ol r}}{2\sqrt{2/\pi}} \leq \frac{1}{2}$, the sharpness property \eqref{L1 RIP property in lemma} implies exact recovery as the left-hand side is equal to $f(\calX) - f(\calX^\star)$. Additionally, this property indicates that we can tolerate up to $m/2$ outliers in the measurements when $\delta_{2 \ol r}$ is sufficiently small. Denote by $\calA_{\calS}$ and $\calA_{\calS^c}$ as the linear operators in $\{\calA_k: k\in\calS \}$ and $\{\calA_k: k\in\calS^c \}$, respectively. One can also use the same analysis of \cite[Proposition 2]{li2020nonconvex} to obtain a similar sharpness by directly using the $\ell_1/\ell_2$-RIP property: assuming that the measurement operators $\calA$ and $\calA_{\calS^c}$ obey $\ell_1/\ell_2$-RIP as in \Cref{L1 RIP}, then we have i.e., $\frac{1}{m}\|\calA(\calX - \calX^\star) -\vs \|_1 - \frac{1}{m}\|\vs\|_1 \geq (2(1 - p_s)(\sqrt{2/\pi} - \delta_{2\overline{r}}) - (\sqrt{2/\pi} + \delta_{2\overline{r}}))\|\calX - \calX^\star\|_F$ with $p_s\leq \frac{1}{2} - \frac{\delta_{2\overline{r}}}{\sqrt{2/\pi} - \delta_{2\overline{r}}}$. Compared to this result, our result provides a more relaxed condition for $\delta_{2\ol r}$ when $p_s$ is fixed, or for $p_s$ when $\delta_{2\ol r}$ is fixed.

\section{Provably Correct Algorithms for Robust TT Recovery}
\label{Robust tensor recovery}

In this section, we develop gradient-based algorithms to recover $\calX^\star$ from corrupted measurements $\vy = \calA(\calX^\star) + \vs$ as described in \eqref{Noisy Definition of tensor sensing L1 loss} by solving \eqref{Loss Function of robust tensor sensing abstract}. Specifically, we introduce two iterative algorithms. The first algorithm, the projected subgradient method (PSubGM), optimizes the entire tensor in each iteration and employs the TT-SVD to project the iterates back to the TT format. To address the challenge of high-order tensors, which can be exponentially large, we then propose the factorized Riemannian subgradient method (FRSubGM). This method, based on the factorization approach, directly optimizes over the factors, reducing storage memory requirements at the expense of a slightly slower convergence rate compared to PSubGM. Finally, we show that the commonly used truncated spectral initialization provides a valid starting point for both PSubGM and FRSubGM.

\subsection{Projected Subgradient Method}

We commence by reiterating the loss function in \eqref{Loss Function of robust tensor sensing abstract}, which seeks to minimize the disparity between the measurements $\vy$ and the linear map of the estimated low-TT-rank tensor $\calX$ as:
\begin{eqnarray}
    \label{Loss Function of robust tensor sensing main paper}
    \begin{split}
    \min_{\mbox{\tiny$\begin{array}{c}
     \calX\in\R^{d_1\times\cdots\times d_N}\\
     \text{rank}(\calX) = (r_1,\dots, r_{N-1})\end{array}$}} f(\calX) = \frac{1}{m}\|\calA(\calX)-\vy\|_1.
    \end{split}
\end{eqnarray}

We solve \eqref{Loss Function of robust tensor sensing main paper} by a Projected SubGradient Method (PSubGM)  with the following iterative updates:
\begin{eqnarray}
    \label{Iterative equ of GDwithTT_SVD_1}
    \calX^{(t+1)} = \text{SVD}_{\vr}^{tt}( \calX^{(t)} - \mu_t\partial f(\calX^{(t)})),
\end{eqnarray}
where $\mu_t$ is the step size,  $\partial f(\calX^{(t)})  =  \frac{1}{m}\sum_{k=1}^m \text{sign}(\<\calA_k,\calX^{(t)}\> - y_k) \calA_k$ is a subgradient\footnote{The definition of  (Fr$\rm \acute{e}$chet) subdifferential \cite{li2020nonconvex} of $f$ at $\calX$ is
\begin{eqnarray}
    \label{Definition of subgradient of any function main paper}
    \partial f(\calX) = \bigg\{\calD\in\R^{d_1\times \cdots \times d_N}: \liminf_{\calX' \to\calX} \frac{f(\calX') -  f(\calX)  -  \<\calD, \calX'  -\calX  \>  }{\|\calX'  -  \calX\|_F} \geq 0  \bigg\},
\end{eqnarray}
where each $\calD\in \partial f(\calX)$ is called a subgradient of $f$ at $\calX$. In general, a nonsmooth function may have multiple subgraidents at certain points. Here, if there exist multiple subgradients, we pick the one with sign function defined as $\text{sign}(x) = \begin{cases}
-1, &  x < 0 \\
0, &  x = 0 \\
1, & x > 0
\end{cases}$, and with abuse of notation, we use $\partial f(\calX)$ to denote this subgradient.} of $f$, and $\text{SVD}_{\vr}^{tt}( \cdot)$ denotes the TT-SVD operation \cite{Oseledets11} that projects a given tensor to a TT format.
Computing the optimal low-TT-rank approximation, in general, is NP-hard \cite{hillar2013most}.  While the TT-SVD is not a nonexpansive projection, when two tensors are
sufficiently close, it can have an improved guarantee that is independent of $N$, distinguishing it from the result in \cite[Corollary 2.4]{Oseledets11}.
\begin{lemma}(\cite[Lemma 26]{cai2022provable})
\label{Perturbation bound for TT SVD}
Let $\calX^\star$ be in TT format with the ranks $(r_1,\dots, r_{N-1})$. For any $\calE\in\R^{d_1\times \cdots \times d_N}$ with $ C_N ||\calE||_F  \leq \underline{\sigma}({\calX^\star})$ for some constant $C_N\geq 500 N$, we have
\begin{eqnarray}
    \label{Perturbation bound for TT SVD1}
    ||\text{SVD}_{\vr}^{tt}(\calX^\star + \calE)-\calX^\star||_F^2\leq||\calE||_F^2+\frac{600N||\calE||_F^3}{\underline{\sigma}({\calX^\star})}.
\end{eqnarray}
\end{lemma}
\Cref{Perturbation bound for TT SVD} implies that when the initialization of PSubGM is close to $\calX^\star$, the perturbation bound of the TT-SVD is independent of the order $N$ due to $\|\calE\|_F\leq \frac{\underline{\sigma}({\calX^\star})}{500N}$. To facilitate analyzing the local convergence of the PSubGM, we first establish the robust regularity condition which has been widely built in contexts such as low-rank matrix recovery \cite{Tu16}, phase retrieval \cite{candes2015phase} and robust subspace learning \cite{zhu2019linearly}. The result is as follows:
\begin{lemma} (Robust regularity condition of $f$ with respect to the full tensor)
\label{Robust Regularity condition of full tensor}
Let the ground truth tensor $\calX^\star$ be in TT format with ranks $\vr = (r_1,\dots, r_{N-1})$. Assume the linear map $\calA$ is a Gaussian measurement operator where $\{\calA_k\}_{k=1}^m$ have i.i.d. standard Gaussian entries. Then, based on the $\ell_1/\ell_2$-RIP and sharpness property,  $f$ satisfies the robust regularity condition in the sense that
\begin{eqnarray}
    \label{the defitinition of regularity condition for TT L1 loss full tensor}
    \< \calX - \calX^\star, \partial f(\calX)\> \geq ((1-2p_s)\sqrt{2/\pi} - \delta_{2\overline{r}})\|\calX - \calX^\star\|_F,
\end{eqnarray}
for any low-TT-rank tensors $\calX$ with ranks $\vr = (r_1,\dots, r_{N-1})$.
\end{lemma}
The proof is given in Appendix \ref{Proof of the Robust regularity condition for full tensor}. This result essentially ensures that at any feasible point $\calX$, the associated negative search direction $\partial f(\calX)$ maintains a positive correlation with the error $-(\calX - \calX^\star)$. This enables subgradient method with an appropriate step size will consistently move the current point closer to the global solution in each update.

In contrast to gradient descent, subgradient method with a constant step size may fail to converge to a critical point of a nonsmooth function, such as the $\ell_1$ loss, even if the function is convex \cite{nedic2001convergence, bertsekas2012incremental, shor2012minimization}. Therefore, to ensure convergence of PSubGM, it is generally necessary to use a diminishing step size \cite{goffin1977convergence, shor2012minimization}.
Based on \Cref{Robust Regularity condition of full tensor}, we analyze the local convergence of the PSubGM with a diminishing step size.
\begin{theorem} (Local linear convergence of PSubGM)
\label{Local Convergence of PGD in the sensing_Theorem}
Let $\calX^\star$ be in TT format with ranks $\vr = (r_1,\dots, r_{N-1})$.
Assume that $\calA$ obeys the $\ell_1/\ell_2$-RIP and sharpness with a constant $\delta_{2\overline{r}}< \frac{1-2p_s-\sqrt{4c/(3+3c)}}{1+\sqrt{4c/(3+3c)}}\sqrt{2/\pi}$ for a positive constant $c<\frac{3(1-2p_s)^2}{1+12p_s - 12p_s^2}$.
Suppose that the PSubGM in \eqref{Iterative equ of GDwithTT_SVD_1} is initialized with $\calX^{(0)}$ satisfying
\begin{eqnarray}
    \label{Local Convergence of PGD in the sensing_Theorem initialization}
    \|\calX^{(0)} - \calX^\star\|_F\leq \frac{c\underline{\sigma}({\calX^\star})}{600N},
\end{eqnarray}
and uses the step size $\mu_t=\lambda q^t$ in \eqref{Iterative equ of GDwithTT_SVD_1}, where $\lambda = \frac{ (1-2p_s)\sqrt{2/\pi} - \delta_{2\overline{r}}}{2(\sqrt{2/\pi} + \delta_{2\overline{r}})^2}\|\calX^{(0)} - \calX^\star\|_F$ and $q = \sqrt{(1+c)(1-\frac{3((1-2p_s)\sqrt{2/\pi} - \delta_{2\overline{r}})^2}{4(\sqrt{2/\pi} + \delta_{2\overline{r}})^2})}$.
Then, the iterates $\{ \calX^{(t)} \}_{t\geq 0}$ generated by the PSubGM will converge linearly to $\calX^\star$:
\begin{eqnarray}
    \label{Local Convergence of PGD in the sensing_Theorem_1}
    \| \calX^{(t)}   - \calX^\star \|_F^2 \leq q^{2t} \|\calX^{(0)} - \calX^\star\|_F^2 .
\end{eqnarray}
\end{theorem}
This proof is provided in {Appendix} \ref{Proof of local convergence of PGD in appendix}. Note that the required initialization \eqref{Local Convergence of PGD in the sensing_Theorem initialization} and the term $1+c$ in \eqref{Local Convergence of PGD in the sensing_Theorem_1} are introduced because of the sub-optimality of the TT-SVD operation. While we present the result using a single choice of $\lambda$ and $q$ for simplicity, a wider range of values can be slightly modified to the arguments, without compromising linear convergence. In practice, these parameters should be carefully tuned to ensure convergence. In order to ensure that the upper bound of recovery error in \eqref{Local Convergence of PGD in the sensing_Theorem_1} is monotonically decreasing, we can choose $\delta_{2\overline{r}}<\frac{1-2p_s-\sqrt{4c/(3+3c)}}{1+\sqrt{4c/(3+3c)}}\sqrt{2/\pi}\leq (1-2p_s)\sqrt{2/\pi}$. This can be guaranteed by choosing sufficiently large $m$ according to \Cref{L1 RIP}.
Additionally, it should be noted that the linear convergence rate of the PSubGM improves when either the outlier ratio $p_s$ decreases or the number of measurements $m$ increases.
On the other hand, unlike the inexact recovery associated with the smooth least-squares ($\ell_2$) loss \cite{Rauhut17} in the presence of outliers, it is important to emphasize that the $\ell_1$ loss enables the precise recovery of $\calX^\star$. This means that under the $\ell_1$ loss, it is possible to recover the true underlying tensor $\calX^\star$ from measurements with outliers. Nonetheless, in the presence of dense noise, the recovery of the ground truth tensor $\calX^\star$ becomes infeasible, regardless of whether the $\ell_1$ or $\ell_2$ loss function is employed.

\subsection{Factorized Riemannian Subgradient Method}

One drawback of PSubGM is that it requires storing the entire tensor ($O(\ol d^N)$ size) in each iteration. To reduce the space complexity, an alternative approach is to directly estimate tensor factors $\{\mX_i \}$, which have a complexity of $O(N\ol d \ol r^2)$, by solving
\begin{eqnarray}
    \label{Loss Function of robust tensor sensing tensor factor main paper}
    \begin{split}
    \min_{\mbox{\tiny$\begin{array}{c}
     \mX_i\in\R^{r_{i-1}\times d_i\times r_i},\\
     i\in[N]\end{array}$}} F(\mX_1,\dots, \mX_N) & = \frac{1}{m}\|\calA([\mX_1,\dots,\mX_N])-\vy\|_1,\\
     &\st \ \sum_{s_i=1}^{d_i} \mX_i^\top(:,s_i,:)\mX_i(:,s_i,:) =\mId_{r_i},  \ \ i\in[N-1].
    \end{split}
\end{eqnarray}
The additional constraints $\sum_{s_i=1}^{d_i} \mX_i^\top(:,s_i,:)\mX_i(:,s_i,:) =\mId_{r_i}$ are introduced to reduce the scaling ambiguity of the factors \cite{qin2024guaranteed}, i.e., recovering the left-orthogonal form.  Noticing that constraints define a Stiefel manifold structure, we apply a Riemannian Subgradient method on the Stiefel manifold \cite{absil2008optimization} to optimize it. We call the resulting algorithm FRSubGM, short for factorized Riemannian Subgradient Method, to emphasize the factorization approach.
Specifically, recalling the left unfolding  $L(\mX_i)$ of factors in \eqref{left unfolding of tensor factor of Tensor Train}, FRSubGM involves iterative updates
\begin{eqnarray}
    \label{RIEMANNIAN_GRADIENT_DESCENT_1_1 L1 loss}
    &&\hspace{-0.8cm} L(\mX_i^{(t+1)})=\text{Retr}_{L(\mX_i)}\bigg(L(\mX_i^{(t)})-\frac{\mu_t}{\ol\sigma^2(\calX^\star)}\calP_{\text{T}_{L({\mX_i})} \text{St}}\big(\partial_{L(\mX_{i})}F(\mX_1^{(t)}, \dots, \mX_N^{(t)})\big) \bigg), i\in[N-1],\\
    \label{RIEMANNIAN_GRADIENT_DESCENT_1_2 L1 loss}
    &&\hspace{-0.8cm} L(\mX_N^{(t+1)})=L(\mX_N^{(t)})-\mu_t\partial_{L(\mX_N)}F(\mX_1^{(t)}, \dots, \mX_N^{(t)}),
\end{eqnarray}
where $\calP_{\text{T}_{L({\mX_i})}}(\mU) = \mU - \frac{1}{2}L({\mX_i})(\mU^\top L({\mX_i}) + (L({\mX_i}))^\top \mU)$ denotes the projection onto the tangent space of the Stifel manifold at the point $L({\mX_i})$ and the polar decomposition-based retraction is $\text{Retr}_{L(\mX_i)}( \mG) = \mG(\mG^\top \mG)^{-\frac{1}{2}}$.
Moreover, $\mu_t$ is a diminishing step size. Note that we use discrepant step sizes between $\{L(\mX_i)\}$ and $L(\mX_N)$, i.e., $\mu_t/\ol\sigma^2(\calX^\star)$ for $\{L(\mX_i)\}$ and $\mu_t$ for $L(\mX_N)$. This is because $\|L_{\mR}({\mX}_i^\star)\|^2 = 1, i\in[N-1]$ and $\|R({\mX}_N^\star)\|^2 = \sigma_1^2({\calX^\star}^{\< N-1 \>})\leq \ol{\sigma}^2(\calX^\star)$, where $R(\mX_N^\star)=\begin{bmatrix}\mX_N^\star(1) &  \cdots &  \mX_N^\star(d_N) \end{bmatrix}\in\R^{r_{N-1}\times d_N}$, are satisfied in each iteration. For simplicity, we use $\ol{\sigma}^2(\calX^\star)$ to unify the step size. However, in practical implementation, we have the flexibility to fine-tune the two step sizes.

Before analyzing the FRSubGM algorithm, we will establish an error metric to quantify the distinctions between factors in two left-orthogonal form tensors, namely   $\calX = [{\mX}_1,\dots,{\mX}_{N} ]$ and  $\calX^\star = [{\mX}_1^\star,\dots,{\mX}_{N}^\star ]$.
Note that the left-orthogonal form still has rotation ambiguity among the factors in the sense that $\Pi_{i=1}^{N}\mX_i^\star(s_i) = \Pi_{i=1}^{N}\mR_{i-1}^\top\mX_i^\star(s_i)\mR_i$ for any orthonormal matrix $\mR_i\in\O^{r_i\times r_i}$ (with $\mR_0 = \mR_{N} = 1$). To capture this rotation ambiguity, by defining the rotated factors $L_{\mR}(\mX_i^\star)$ as
\begin{eqnarray}
\label{A modified left unfolding}
    L_{\mR}(\mX_i^\star) = \begin{bmatrix}\mR_{i-1}^\top\mX_{i}^\star(1)\mR_i\\ \vdots \\ \mR_{i-1}^\top\mX_{i}^\star(d_i)\mR_i\end{bmatrix}, \ \forall \mR_{i-1}\  \text{and} \  \mR_i\in\O^{r_i\times r_i},
\end{eqnarray}
we then apply the distance  between the two sets of factors as \cite{qin2024guaranteed}
\begin{eqnarray}
\label{BALANCED NEW DISTANCE BETWEEN TWO TENSORS}
    \text{dist}^2(\{\mX_i\},\{ \mX_i^\star\})=\min_{\mR_i\in\O^{r_i\times r_i}, \atop i \in [N-1]}\sum_{i=1}^{N-1} \ol{\sigma}^2(\calX^\star)\|L({\mX}_i)-L_{\mR}({\mX}_i^\star)\|_F^2 + \|L({\mX}_N)-L_{\mR}({\mX}_N^\star)\|_2^2,
\end{eqnarray}
where we note that $L({\mX}_N), L_{\mR}({\mX}_N^\star)\in\R^{(r_{N-1}d_N)\times 1}$ are vectors. Subsequently, we establish a connection between $\text{dist}^2(\{\mX_i\},\{ \mX_i^\star\})$ and $\|\calX-\calX^\star\|_F^2$, implying the convergence behavior of $\|\calX-\calX^\star\|_F^2$ as $\{\mX_i\}$ approaches global minima.
\begin{lemma}(\cite[Lemma 1]{qin2024guaranteed})
\label{LOWER BOUND OF TWO DISTANCES}
For any two TT format tensors $\calX$ and $\calX^\star $ with ranks $\vr = (r_1,\dots, r_{N-1})$ and $\ol{\sigma}^2(\calX)\leq \frac{9\ol{\sigma}^2(\calX^\star)}{4}$, let $\{\mX_i\}$ and $\{\mX_i^\star\}$ be the corresponding left-orthogonal form factors. Then $\|\calX-\calX^\star\|_F^2$ and $\dist^2(\{\mX_i \},\{ \mX_i^\star \})$ defined in \eqref{BALANCED NEW DISTANCE BETWEEN TWO TENSORS} satisfy
\begin{eqnarray}
    \label{LOWER BOUND OF TWO DISTANCES_1 main paper}
    &&\|\calX-\calX^\star\|_F^2\geq\frac{1}{8(N+1+\sum_{i=2}^{N-1}r_i)\kappa^2(\calX^\star)}\dist^2(\{\mX_i \},\{ \mX_i^\star \}),\\
    \label{UPPER BOUND OF TWO DISTANCES_1 main paper}
    &&\|\calX-\calX^\star\|_F^2\leq\frac{9N}{4}\dist^2(\{\mX_i \},\{ \mX_i^\star \}).
\end{eqnarray}
\end{lemma}
\Cref{LOWER BOUND OF TWO DISTANCES} ensures that $\calX$ is close to $\calX^\star$ once the corresponding factors are close with respect to the proposed distance measure, and the convergence behavior of $\|\calX-\calX^\star\|_F^2$ is reflected by the convergence in terms of the factors. Next, we first provide the robust regularity condition of $F(\mX_1,\dots, \mX_N)$.

\begin{lemma} (Robust regularity condition of $F$ with respect to tensor factors)
\label{Robust Regularity condition of factor tensor}
Let the ground truth tensor $\calX^\star$ be in TT format with ranks $\vr = (r_1,\dots, r_{N-1})$. Assume the linear map $\calA$ obeys the $\ell_1/\ell_2$-RIP and sharpness with a constant $\delta_{(N+1)\overline{r}}\leq(1-2p_s)\sqrt{2/\pi}$.
Define the set $\calC(b)$ as
\begin{eqnarray}
    \label{the defitinition of the intial set of L1 loss in TT tensor factors}
    \calC(b) : =\bigg\{\calX: \|\calX - \calX^\star\|_F^2\leq b \bigg\},
\end{eqnarray}
where $b = \frac{\underline{\sigma}^4(\calX^\star) ((1-2p_s)\sqrt{2/\pi} -  \delta_{(N+1)\overline{r}})^2 }{144(2N^2 - 2N +1)(N+1+\sum_{i=2}^{N-1}r_i)^2(\sqrt{2/\pi} + \delta_{(N+1)\overline{r}})^2\ol{\sigma}^2(\calX^\star)}$. Then for any TT format $\calX \in\calC(b)$, $F$ satisfies the robust regularity condition:
\begin{eqnarray}
    \label{the defitinition of regularity condition for TT L1 loss  tensor factors}
    &\!\!\!\!\!\!\!\!&\sum_{i=1}^{N} \bigg\< L(\mX_i)-L_{\mR}(\mX_i^\star),\calP_{\text{T}_{L(\mX_i)} \text{St}}\bigg(\partial_{L(\mX_{i})}F(\mX_1, \dots, \mX_N)\bigg)\bigg\> \nonumber\\
    &\!\!\!\!\geq\!\!\!\!& \frac{(1-2p_s)\sqrt{2/\pi} - \delta_{(N+1)\overline{r}}}{4\sqrt{2(N+1+\sum_{i=2}^{N-1}r_i)}\kappa(\calX^\star)}\text{dist}^2(\{\mX_i\},\{ \mX_i^\star\}).
\end{eqnarray}
To simplify the expression, we define the identity operator $\calP_{\text{T}_{L({\mX}_N)} \text{St}}=\calI$ such that $\calP_{\text{T}_{L({\mX}_N)} \text{St}}(\nabla_{L({\mX}_{N})}F(\mX_1, \dots, \mX_N)) = \nabla_{L({\mX}_{N})}F(\mX_1, \dots, \mX_N)$.
\end{lemma}
The proof is shown in {Appendix} \ref{Proof of the Robust regularity condition for factor tensor}. This result guarantees a positive correlation between the errors $-\{ L(\mX_i)-L_{\mR}(\mX_i^\star)\}_{i=1}^N$ and negative Riemannian search directions $\{\calP_{\text{T}_{L(\mX_i)} \text{St}}(\partial_{L(\mX_{i})}F(\mX_1, \dots, \mX_N))\}_{i=1}^N$ for $N$ factors in the Riemannian space, i.e., negative Riemannian direction points towards the true factors. When initialed properly, we then obtain a linear convergence of the FRSubGM with a diminishing step size as following:
\begin{theorem} (Local linear convergence of the FRSubGM)
\label{Local Convergence of SGD in the l1 sensing_Theorem}
Let the ground truth tensor $\calX^\star$ be in TT format with ranks $\vr = (r_1,\dots, r_{N-1})$. Assume that the linear map $\calA$ obeys the $\ell_1/\ell_2$-RIP and sharpness with a constant $\delta_{(N+1)\overline{r}}\leq(1-2p_s)\sqrt{2/\pi}$. Suppose that the FRSubGM in \eqref{RIEMANNIAN_GRADIENT_DESCENT_1_1 L1 loss} and \eqref{RIEMANNIAN_GRADIENT_DESCENT_1_2 L1 loss} is initialized with $\calX^{(0)}$ satisfying $\calX^{(0)}\in\calC(b)$. In addition, we set the step size $\mu_t = \lambda q^t$ with $\lambda=\frac{(1-2p_s)\sqrt{2/\pi} - \delta_{(N+1)\overline{r}}}{\sqrt{2(N+1+\sum_{i=2}^{N-1}r_i)}(9N - 5)(\sqrt{2/\pi} + \delta_{(N+1)\overline{r}})^2\kappa(\calX^\star)}\text{dist}(\{\mX_i^{(0)}\},\{ \mX_i^\star\})$ and $q= \sqrt{1- \frac{ ((1-2p_s)\sqrt{2/\pi} -  \delta_{(N+1)\overline{r}})^2 }{8(N+1+\sum_{i=2}^{N-1}r_i)(9N - 5)(\sqrt{2/\pi} + \delta_{(N+1)\overline{r}})^2\kappa^2(\calX^\star)}}$. Then, the iterates $\{\mX_i^{(t)}\}_{t\geq 0}$ generated by the FRSubGM will converge linearly to $\{ \mX_i^\star\}$ (up to rotation):
\begin{eqnarray}
    \label{Local Convergence of SGD in the L1 sensing_Theorem_1}
    \text{dist}^2(\{\mX_i^{(t)}\},\{ \mX_i^\star\})\leq \text{dist}^2(\{\mX_i^{(0)}\},\{ \mX_i^\star\})q^{2t}.
\end{eqnarray}
\end{theorem}
The proof is provided in {Appendix} \ref{Local Convergence Proof of SGD l1 Tensor Sensing}. It is important to note that the convergence rate of FRSubGM is still linear, but the convergence rate of FRSubGM depends not only on the values of $\delta_{(N+1)\overline{r}}$ and $p_s$, but also on the ratio $\frac{1}{\kappa^2(\calX^\star)}$ and the parameter $N$. Consequently, the convergence rate of FRSubGM could be slower than that of PSubGM. In addition, according to \Cref{LOWER BOUND OF TWO DISTANCES}, we can also derive the linear convergence in terms of the entire tensor, i.e., $\|\calX^{(t)} - \calX^\star \|_F \leq O(\frac{\underline{\sigma}^2(\calX^\star)}{N^2 \ol r}) q^{2t}$. Ultimately, even though it may not be straightforward to choose exact deterministic values for $\lambda$ and $q$ in practice, we can still select values that are close to these desired values.

\subsection{Truncated spectral initialization}
The above local linear convergence for both PSubGM and FRSubGM requires an appropriate initialization. To achieve such an initialization, the spectral initialization method is commonly employed in the literature \cite{candes2015phase,luo2019optimal,cai2019nonconvex,Han20, qin2024guaranteed}. In the presence of outliers, we employ the truncated spectral initialization method \cite{zhang2018median,li2020non,tong2022accelerating}:
\begin{eqnarray}
    \label{Noisy TSPECTRAL INITIALIZATION}
    \calX^{(0)}=\text{SVD}_{\vr}^{tt}\bigg(\frac{1}{(1-p_s)m}\sum_{k=1}^my_k\calA_k \setI_{\{ |y_k|\leq |\vy|_{(\lceil p_s m \rceil)}  \}}  \bigg),
\end{eqnarray}
where $|\vy|_{(k)}$ denotes the $k$-th largest amplitude of $\vy$ and $\setI_{\{ |y_k|\leq |\vy|_{(\lceil p_s m \rceil)}  \}}$ indicates that this term is $1$ if $|y_k|\leq |\vy|_{(\lceil p_s m \rceil)}$ and $0$ otherwise. Recall that $\text{SVD}_{\vr}^{tt}(\cdot )$ is the TT-SVD algorithm for finding a TT approximation.

The following result ensures that such an initialization $\calX^{(0)}$ provides a good approximation of $\calX^\star$.
\begin{theorem}
\label{Error analysis of truncated spectral initialization}
Let the ground truth tensor $\calX^\star$ be in TT format with ranks $\vr = (r_1,\dots, r_{N-1})$. Suppose the linear map $\calA: \R^{d_{1}\times  \cdots \times d_{N}}\rightarrow \R^m$ is a Gaussian measurement operator where $\{\calA_k\}_{k=1}^m$ have i.i.d. standard Gaussian entries. Then with probability at least $1-e^{-\Omega(N\ol d\ol r^2 \log N)} - e^{-\Omega(\log ((1 - p_s) m))}$, the spectral initialization $\calX^{(0)}$ generated by \eqref{Noisy TSPECTRAL INITIALIZATION} satisfies
\begin{eqnarray}
    \label{upper bound of truncated spectral initialization main paper}
    \|\calX^{(0)} - \calX^\star \|_F \leq O\bigg(\frac{N\ol r \log ((1 - p_s) m) \|\calX^\star\|_F\sqrt{ \ol d \log N } }{\sqrt{(1-p_s)m}}\bigg).
\end{eqnarray}
\end{theorem}
The proof is provided in {Appendix} \ref{proof of truncated spectral initialization}. In summary, \Cref{Error analysis of truncated spectral initialization} indicates that a sufficiently large $m$ allows for the identification of a suitable initialization that is appropriately close to the ground truth. Additionally, although $p_s$ is specified in \eqref{Noisy TSPECTRAL INITIALIZATION}, it is generally unknown in practical scenarios. Therefore, any constant $\alpha\in[0,1]$ can substitute for $p_s$ in $\lceil p_s m \rceil$. It should be noted that while a smaller $\alpha$ may eliminate more measurements containing outliers, this might necessitate a larger number of measurements $m$ to achieve a satisfactory initialization.

\section{Numerical Experiments}
\label{Numerical experiments}

In this section, we conduct  numerical experiments to evaluate the performance of PSubGM and  FRSubGM algorithms in robust TT recovery.
We generate an order-$N$ ground truth tensor $\calX^\star\in\R^{d_1\times\cdots \times d_N}$ with ranks $\vr = (r_1,\dots, r_{N-1})$ by truncating a random Gaussian tensor using a sequential SVD, followed by normalizing it to unit Frobenius norm. To simplify the selection of parameters, we let $d=d_1=\cdots=d_N$ and $r=r_1=\cdots = r_{N-1}$. We then obtain measurements $\{y_k\}_{k=1}^m$ in \eqref{Noisy Definition of tensor sensing L1 loss} from measurement operator $\calA_k$ which is a random tensor with independent entries generated from the normal distribution, ensuring that $\<\calA_k,\calX^\star  \>\sim\calN(0,1)$, and outlier $s_k \sim\calN(0,10), k\in\calS$ where $|\calS| = p_s m$. The elements in $\calS$ are randomly selected from the set $\{1,\dots, m \}$. We conduct 20 Monte Carlo trials and take the average over the 20 trials.

\begin{figure}[!ht]
\centering
\subfigure[]{
\begin{minipage}[t]{0.31\textwidth}
\centering
\includegraphics[width=5.5cm]{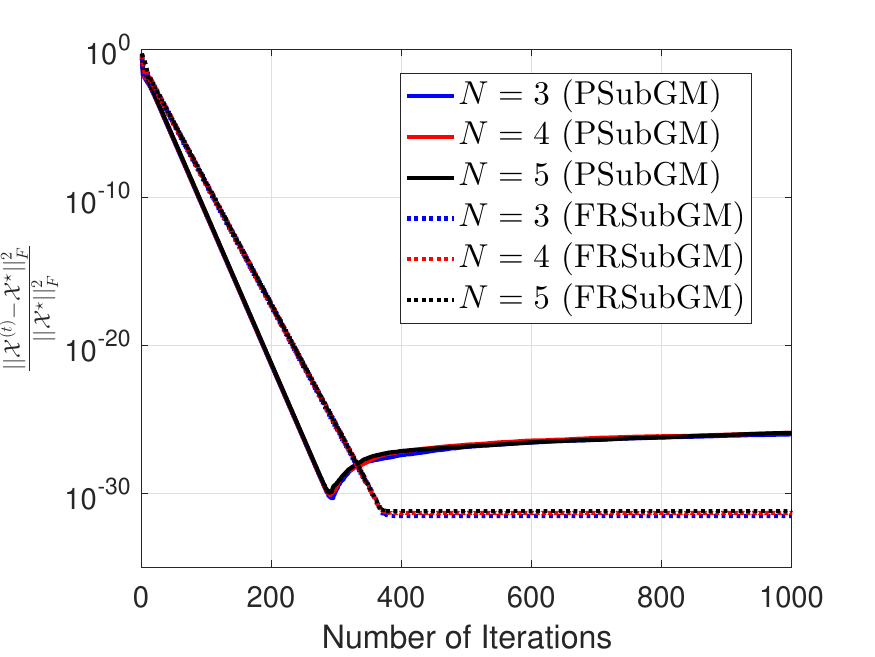}
\end{minipage}
\label{TT_sensing robust_N}
}
\subfigure[]{
\begin{minipage}[t]{0.31\textwidth}
\centering
\includegraphics[width=5.5cm]{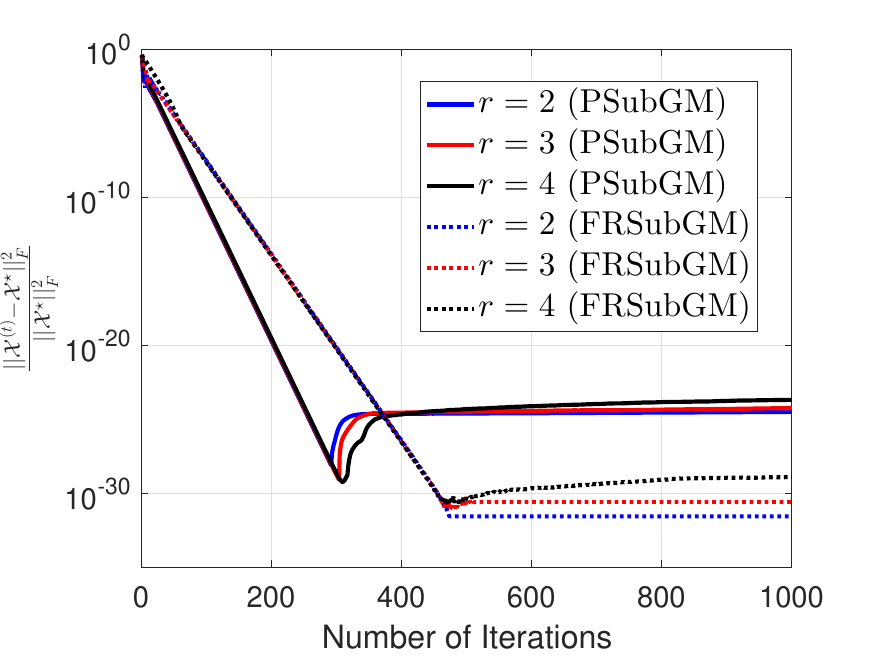}
\end{minipage}
\label{TT_sensing robust_r}
}
\subfigure[]{
\begin{minipage}[t]{0.31\textwidth}
\centering
\includegraphics[width=5.5cm]{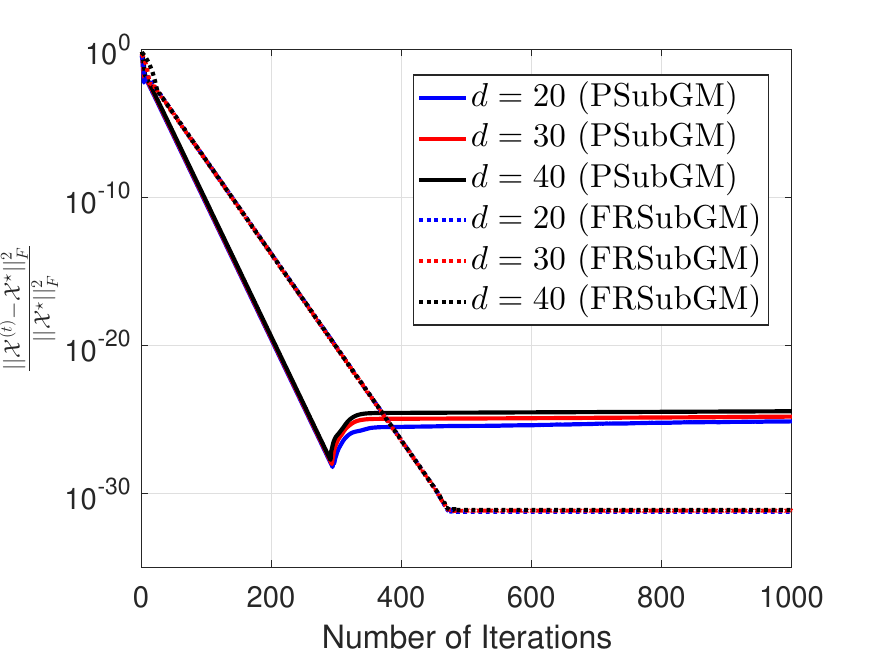}
\end{minipage}
\label{TT_sensing robust_d}
}
\caption{Performance comparison of the PSubGM and FRSubGM in the robust tensor recovery, (a) for different $N$ with $d = 6$, $r = 2$, $m = 1500$, $p_s = 0.3$, $\lambda = 0.5$, $q = 0.9$ (PSubGM) and $q = 0.91$ (FRSubGM)  (b) for different $r$ with $d = 40$, $N = 3$, $m = 12000$, $p_s = 0.3$, $\lambda = 0.5$, $q = 0.9$ (PSubGM) and  $q = 0.93$ (FRSubGM) (c) for different $d$ with $N = 3$, $r= 2$, $m = 3000$, $p_s = 0.3$,  $\lambda = 0.5$, $q = 0.9$ (PSubGM) and  $q = 0.93$ (FRSubGM).}
\end{figure}

\begin{figure}[!ht]
\centering
\subfigure[]{
\begin{minipage}[t]{0.45\textwidth}
\centering
\includegraphics[width=5cm]{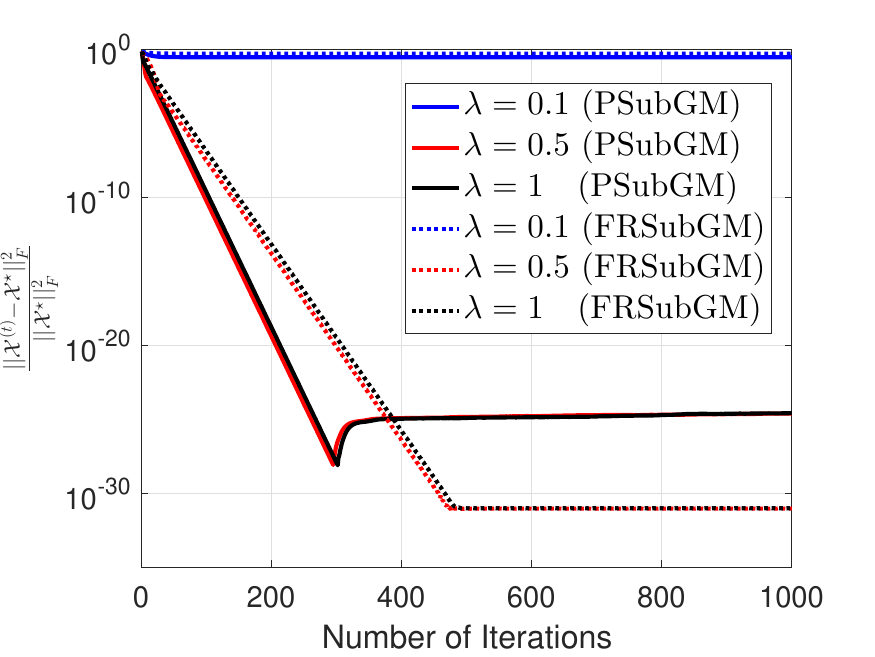}
\end{minipage}
\label{TT_sensing robust_lambda}
}
\subfigure[]{
\begin{minipage}[t]{0.45\textwidth}
\centering
\includegraphics[width=5cm]{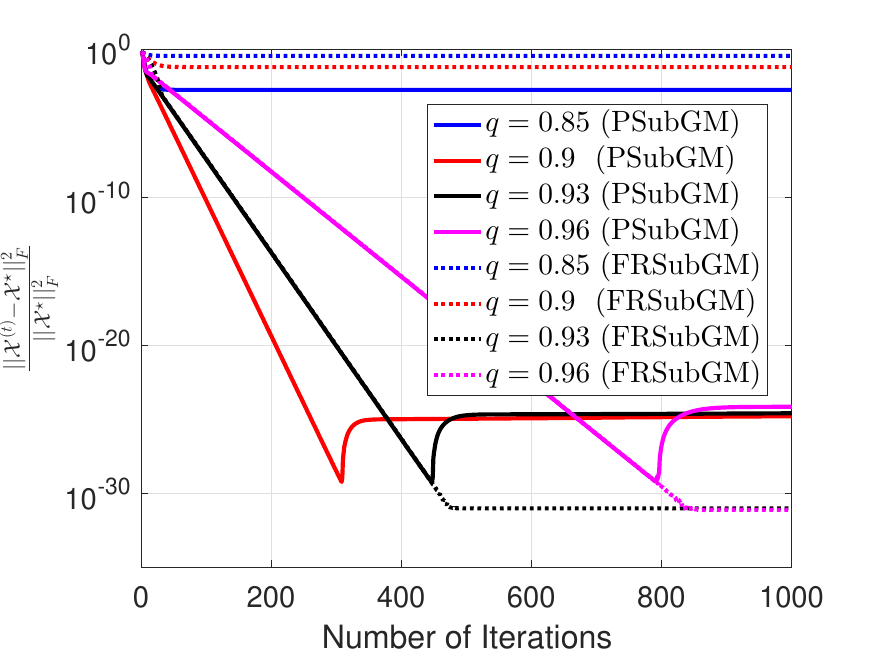}
\end{minipage}
\label{TT_sensing robust_q}
}
\caption{Performance comparison of the PSubGM and FRSubGM in the robust tensor recovery, (a) for different $\lambda$ with $N = 3$, $d = 30$, $r = 2$, $m = 2000$, $p_s = 0.3$, $q = 0.9$ (PSubGM) and $q = 0.93$ (FRSubGM), (b) for different $q$ with $N = 3$, $d = 30$, $r = 2$, $m = 2000$, $p_s = 0.3$ and $\lambda = 0.5$.}
\end{figure}

\begin{figure}[!ht]
\centering
\includegraphics[width=6.5cm, keepaspectratio]%
{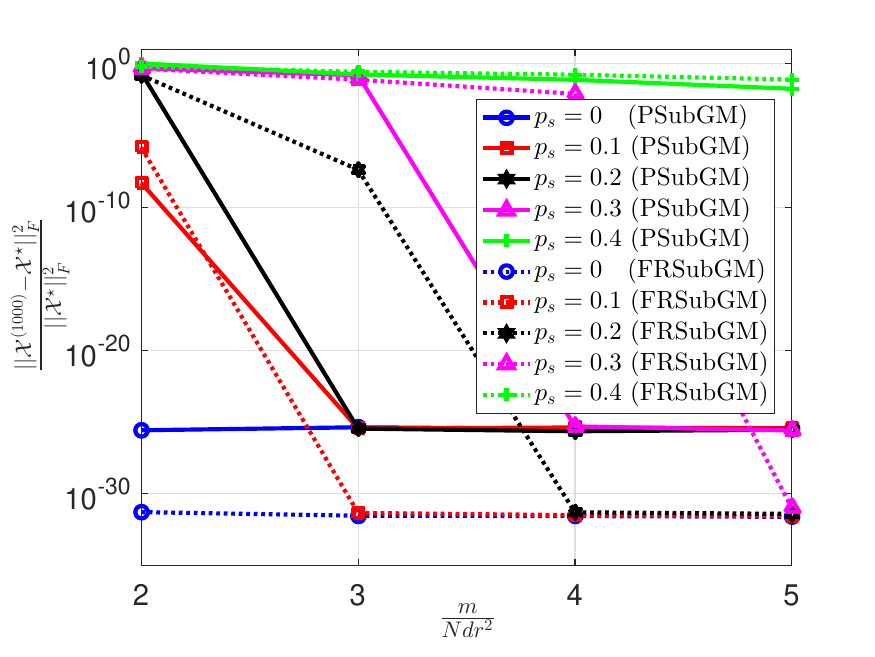}
\caption{Recovery error of the PSubGM and FRSubGM for different $p_s$ and $m$ with $N = 3$, $d = 10$, $r = 2$, $\lambda = 0.5$, $q = 0.9$ (PSubGM) and $q = 0.91$ (FRSubGM).}
\label{The error bound of TT sensing with different_m and ps}
\end{figure}

In the first experiment, we evaluate the performance of PSubGM and FRSubGM in terms of $N$, $d$, and $r$. Figures~\ref{TT_sensing robust_N}-\ref{TT_sensing robust_d} clearly demonstrate that PSubGM exhibits a faster convergence speed compared to FRSubGM. However, the final recovery error of PSubGM is slightly higher than that of FRSubGM, which can possibly be attributed to the sub-optimality of the TT-SVD. Notably, unlike the significantly slower convergence rates of IHT \cite{Rauhut17} and FRGD \cite{qin2024guaranteed} as $N$, $r$, and $d$ increase observed from \cite[Figure 1]{qin2024computational}, PSubGM and FRSubGM do not exhibit such degradation, attributable to the use of a diminishing step size.

In the second experiment, we test the performance of PSubGM and  FRSubGM in terms of $\lambda$ and $q$. In Figures~\ref{TT_sensing robust_lambda} and \ref{TT_sensing robust_q}, we can observe that both larger and smaller values of $\lambda$ or $q$ can potentially result in slower convergence or worse recovery error. Therefore, it is crucial to carefully fine-tune the parameters $\lambda$ and $q$ to ensure optimal performance.

In the third experiment, we test the performance of PSubGM and  FRSubGM in terms of $p_s$ and $m$. Figure~\ref{The error bound of TT sensing with different_m and ps} illustrates the relationship between recovery error and $p_s$ and $m$. It is evident that a larger value of $\frac{m}{Ndr^2}$ ensures better performance, as the $\ell_1/\ell_2$-RIP constant $\delta$ is inversely proportional to $m$. Moreover, as $p_s$ increases, a larger number of measurement operators $m$ is required.

In the last experiment, we investigate the necessary value of $m$ for different $N$ utilizing PSubGM and FRSubGM. This investigation is illustrated in Figures~\ref{TT_sensing robust_different N and m PSubGM} and \ref{TT_sensing robust_different N and m RSubGM}, where we evaluate the success rate of achieving $\frac{\|\calX^{(1000)} - \calX^\star \|_F^2}{\|\calX^\star\|_F^2}\leq10^{-5}$ over $100$ independent trials.  Our findings demonstrate a trend: as the value of $m$ increases and $N$ decreases, the success rate of recovery also improves. In addition, we establish a linear correlation between the number of measurements, $m$, and the tensor oders $N$, aligning with the conditions stipulated in \Cref{L1 RIP}. It is important to recognize that FRSubGM necessitates a larger value of $m$ compared to PSubGM. This difference arises from our analysis, where FRSubGM must adhere to the $(N+1) \ol r$-$\ell_1/\ell_2$-RIP, whereas PSubGM only requires the $3 \ol r$-$\ell_1/\ell_2$-RIP.

\begin{figure}[!ht]
\centering
\subfigure[]{
\begin{minipage}[t]{0.47\textwidth}
\centering
\includegraphics[width=5.8cm]{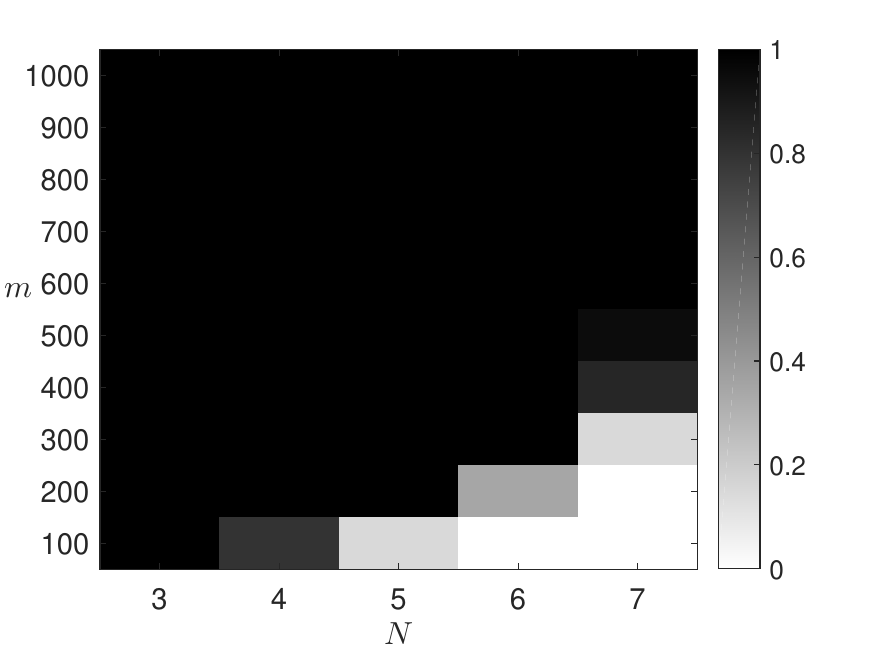}
\end{minipage}
\label{TT_sensing robust_different N and m PSubGM}
}
\subfigure[]{
\begin{minipage}[t]{0.47\textwidth}
\centering
\includegraphics[width=5.8cm]{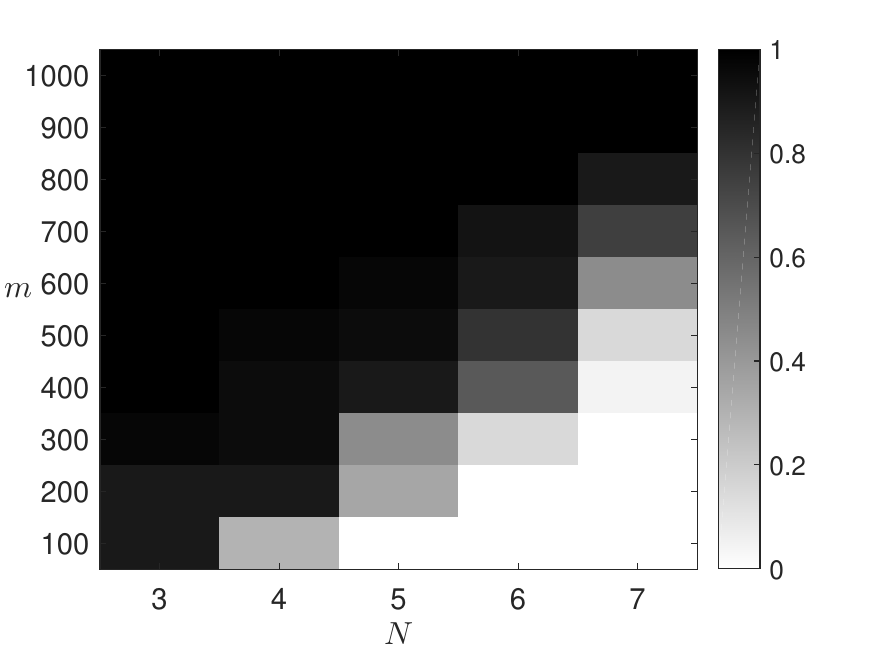}
\end{minipage}
\label{TT_sensing robust_different N and m RSubGM}
}
\caption{Performance comparison of (a) PSubGM and (b) FRSubGM in the robust tensor recovery, for different $N$ and $m$ with $d=4$, $r=2$, $p_s = 0.05$, $\lambda = 0.07$ and $q=0.99$.}
\end{figure}

\section{Conclusion}
\label{conclusion}

In this paper, we develop efficient algorithms with guaranteed performance for robust tensor train (TT) recovery in the presence of outliers. We first prove the $\ell_1/\ell_2$-RIP of the Gaussian measurement operator and the sharpness property of the robust $\ell_1$ loss formulation, implying the possibility of exact recovery even when the measurements are corrupted by outliers. We then propose two iterative algorithms, namely the projected subgradient method (PSubGM) and the factorized Riemannian subgradient method (FRSubGM), to solve the corresponding recovery problems. With suitable initialization and diminishing step sizes, we show that both PSubGM and FRSubGM converge to the ground truth tensor at a linear rate and can tolerate a significant amount of outliers, with the outlier ratio being as large as half. We also demonstrate that a truncated spectral method can provide an appropriate initialization to ensure the local convergence of both algorithms.

As mentioned earlier, the convergence rate of the FRSubGM  is influenced by the condition number $\kappa(\calX^\star)$; as $\kappa(\calX^\star)$ increases, the convergence rate becomes slower. In future research, it would be worthwhile to incorporate the scaled technique \cite{tong2021accelerating, tong2021low, TongTensor21} into our algorithms to mitigate the impact of $\calX^\star$ on convergence. Another promising avenue for future exploration is the investigation of robust overparameterized TT recovery, building upon the advancements made in the matrix case \cite{stoger2021small, jiang2022algorithmic, ding2022validation, xu2023power}.

\section*{Acknowledgment}
We acknowledge funding support from NSF Grants No.\ CCF-2241298 and ECCS-2409701. We thank the Ohio Supercomputer Center for providing the computational resources needed in carrying out this work.

\appendices

\section{Technical tools used in proofs}
\label{Technical tools used in proofs}

We present some useful results for the proofs in the next sections.

\begin{lemma}(\cite[Lemma 3]{qin2024guaranteed})
\label{EXPANSION_A1TOAN-B1TOBN_1}
For any $\mA_i,\mA^\star_i\in\R^{r_{i-1}\times r_i},i\in[N]$, we have
\begin{eqnarray}
    \label{EXPANSION_A1TOAN-B1TOBN_2}
    \mA_1\mA_2\cdots \mA_N-\mA_1^\star\mA_2^\star\cdots \mA_N^\star = \sum_{i=1}^N \mA_1^\star \cdots \mA_{i-1}^\star (\mA_{i} - \mA_i^\star) \mA_{i+1} \cdots \mA_N.
\end{eqnarray}

\end{lemma}

\begin{lemma}
\label{Upper bound of subgradient of f}
Consider the loss function $f(\calX)  = \frac{1}{m}\|\calA(\calX)-\vy\|_1$, where the measurement operator $\calA$ satisfies the $\ell_1/\ell_2$-RIP with a constant $\delta_{2\overline{r}}$. Then for any $\calX$, it holds that
\begin{eqnarray}
    \label{The derivation of subgradient upper bound conclusion}
    \|\calD\|_F\leq \sqrt{2/\pi} + \delta_{2\overline{r}}, \forall \calD\in \partial f(\calX).
\end{eqnarray}
\end{lemma}

\begin{proof}
Recall the definition of  (Fr$\rm \acute{e}$chet) subdifferential \cite{li2020nonconvex} of $f$ at $\calX$
\begin{eqnarray}
    \label{Definition of subgradient of any function}
    \partial f(\calX) = \bigg\{\calD\in\R^{d_1\times \cdots \times d_N}: \liminf_{\calX' \to\calX} \frac{f(\calX') -  f(\calX)  -  \<\calD, \calX'  -\calX  \>  }{\|\calX'  -  \calX\|_F} \geq 0  \bigg\},
\end{eqnarray}
where each $\calD\in \partial f(\calX)$ is called a subgradient of $f$ at $\calX$.

Now for any $\calX'\in\R^{d_1\times \cdots \times d_N}$, we have
\begin{eqnarray}
    \label{The derivation of subgradient upper bound 1}
    | f(\calX') - f(\calX)  | &\!\!\!\!=\!\!\!\!& \frac{1}{m}\big| \| \calA(\calX') - \vy\|_1 - \| \calA(\calX) - \vy\|_1   \big|\nonumber\\
    &\!\!\!\!\leq\!\!\!\!&\frac{1}{m} \|\calA(\calX' - \calX) \|_1\nonumber\\
    &\!\!\!\!\leq\!\!\!\!& (\sqrt{2/\pi} + \delta_{2\overline{r}}) \| \calX' - \calX\|_F,
\end{eqnarray}
where the second inequality follows from the $\ell_1/\ell_2$-RIP of $\calA$. This further implies that
\begin{eqnarray}
    \label{The derivation of subgradient upper bound 2}
    \liminf_{\calX' \to\calX} \frac{| f(\calX') - f(\calX)|}{\|\calX' - \calX\|_F}\leq \lim_{\calX' \to\calX}\frac{(\sqrt{2/\pi} + \delta_{2\overline{r}}) \| \calX' - \calX\|_F}{\| \calX' - \calX\|_F} = \sqrt{2/\pi} + \delta_{2\overline{r}}.
\end{eqnarray}
Upon taking $\calX' = \calX + t \calD, t\to 0$ and  invoking  \eqref{Definition of subgradient of any function}, we have
\begin{eqnarray}
    \label{The derivation of subgradient upper bound 3}
    \|\calD\|_F\leq \sqrt{2/\pi} + \delta_{2\overline{r}}, \ \ \forall \calD\in \partial f(\calX).
\end{eqnarray}

\end{proof}

\begin{lemma}(\cite[Lemma 1]{LiSIAM21})
\label{NONEXPANSIVENESS PROPERTY OF POLAR RETRACTION_1}
Let $\mX^\top\mX = \mId$ and ${\bm \xi}$ on the tangent space of Stiefel manifold be given. Consider the point ${\mX}^+={\mX}+{\bm \xi}$. Then, the polar decomposition-based retraction satisfies $\text{Retr}_{\mX}({\mX}^+)={\mX}^+({{\mX}^+}^\top{\mX}^+)^{-\frac{1}{2}}$ and
\begin{eqnarray}
    \label{NONEXPANSIVENESS PROPERTY OF POLAR RETRACTION_2}
    \|\text{Retr}_{\mX}(\mX^+)-\overline{\mX}\|_F\leq\|{\mX}^+-\overline{\mX}\|_F=\|{\mX}+{\bm \xi}-\overline{\mX}\|_F
\end{eqnarray}
for any $\overline{\mX}^\top \overline{\mX} = \mId$.
\end{lemma}

Finally, we introduce a new operation related to the multiplication of submatrices within the left unfolding matrices $L(\mX_i)=\begin{bmatrix}\mX_i^\top(1)  \cdots  \mX_i^\top(d_i) \end{bmatrix}^\top\in\R^{(r_{i-1}d_i) \times r_i}, i \in [N]$. For simplicity, we will only consider the case $d_i=2$, but extending to the general case is straightforward. In particular, let $\mA=\begin{bmatrix}\mA_1 \\ \mA_2 \end{bmatrix}$ and $\mB=\begin{bmatrix}\mB_1 \\ \mB_2 \end{bmatrix}$ be two block matrices, where $\mA_i\in\R^{r_1\times r_2}$ and $\mB_i\in\R^{r_2\times r_3}$ for $i=1,2$. We introduce the notation $\ol \otimes$ to represent the Kronecker product between submatrices in the two block matrices, as an alternative to the standard Kronecker product based on element-wise multiplication.
Specifically, we define $\mA\ol \otimes\mB$ as $\mA\ol \otimes\mB=\begin{bmatrix}\mA_1 \\ \mA_2 \end{bmatrix}\ol \otimes\begin{bmatrix}\mB_1 \\ \mB_2 \end{bmatrix}=\begin{bmatrix}(\mA_1\mB_1)^\top \ (\mA_2\mB_1)^\top \  (\mA_1\mB_2)^\top \  (\mA_2\mB_2)^\top \end{bmatrix}^\top$.

According to \cite[Lemma 2]{qin2024guaranteed}, we can conclude that for any  left-orthogonal TT format tensor  $\calX^\star = [\mX_1^\star,\dots, \mX_N^\star]$, we have
\begin{eqnarray}
    \label{KRONECKER PRODUCT VECTORIZATION11}
    &&\|\calX^\star\|_F = \|\text{vec}(\calX^\star)\|_2=\|L(\mX_1^\star) \ol \otimes \cdots \ol \otimes L(\mX_N^\star)\|_2 = \|L(\mX_N^\star)\|_2, \\
    \label{KRONECKER PRODUCT VECTORIZATION11 - 1}
    &&\|L(\mX_i^\star) \ol \otimes \cdots \ol \otimes L(\mX_{N-1}^\star) \ol \otimes L(\mX_N^\star)\|_2\leq \Pi_{l = i}^{N-1}\|L(\mX_l^\star)\| \|L(\mX_N^\star)\|_2 = \|L(\mX_N^\star)\|_2, \ \ \forall i\in[N-1], \\
    \label{KRONECKER PRODUCT VECTORIZATION11 - 2}
    &&\|L(\mX_i^\star) \ol \otimes \cdots  \ol \otimes L(\mX_j^\star)\|\leq \Pi_{l = i}^{j}\|L(\mX_l^\star)\| = 1, \ \ i\leq j, \ \ \forall i, j\in[N-1], \\
        \label{KRONECKER PRODUCT VECTORIZATION11 - 3}
    &&\|L(\mX_i^\star) \ol \otimes \cdots  \ol \otimes L(\mX_j^\star)\|_F\leq \Pi_{l = i}^{j-1}\|L(\mX_l^\star)\| \|L(\mX_j^\star)\|_F, \ \ i\leq j, \ \ \forall i, j\in[N-1].
\end{eqnarray}

\section{Proof of \Cref{L1 RIP}}
\label{Proof of L1 RIP in appendix}

\begin{proof}
We first compute the covering number for any low-TT-rank tensor $\calX = [\mX_1,\dots, \mX_N]\in\R^{d_1\times \cdots \times d_N}$ with ranks $(r_1,\dots, r_{N-1})$. Given that any TT format can be converted to its left-orthogonal form, we denote $\calX = [\mX_1,\dots, \mX_N]$ as the left-orthogonal format.    According to \cite{zhang2018tensor}, we can construct $\epsilon$-net $\{L(\mX_i^{(1)}), \dots, L(\mX_i^{(n_i)})  \},i\in[N-1]$  for each set of matrices $\{L(\mX_i)\in\R^{d_ir_{i-1}\times r_i},
i\in[N-1]: \|L(\mX_i)\|\leq 1\}$ ($r_0=1$) such that
\begin{eqnarray}
    \label{L1 ProofOf<H,X>forSubGaussian_proof1}
    \sup_{L(\mX_i): \|L(\mX_i)\|\leq 1}\min_{p_i\leq n_i} \|L(\mX_i)-L(\mX_i^{(p_i)})\|\leq \epsilon,
\end{eqnarray}
with the covering number $n_i\leq (\frac{4+\epsilon}{\epsilon})^{d_ir_{i-1}r_i}$.
Also, we can construct $\epsilon$-net $\{ L(\mX_N^{(1)}), \dots, L(\mX_N^{(n_N)}) \}$ for $\{L(\mX_N)\in\R^{d_Nr_{N-1}\times 1}: \|L(\mX_N)\|_2\leq 1  \}$ such that
\begin{eqnarray}
    \label{L1 ProofOf<H,X>forSubGaussian_proof2}
    \sup_{L(\mX_N): \|L(\mX_N)\|_2\leq 1}\min_{p_N\leq n_N} \|L(\mX_N)-L(\mX_N^{(p_N)})\|_2\leq \epsilon,
\end{eqnarray}
with the covering number $n_N\leq (\frac{2+\epsilon}{\epsilon})^{d_Nr_{N-1}r_N}$.
Hence, for any low-rank TT format $\calX$ with $\|\calX\|_F \leq \|\mX_N\|_F \leq 1$ derived from \eqref{KRONECKER PRODUCT VECTORIZATION11 - 1}, its covering argument is $\Pi_{i=1}^N n_i \leq (\frac{4+\epsilon}{\epsilon})^{d_1r_1+\sum_{i=2}^{N-1}d_ir_{i-1}r_i+d_Nr_{N-1}} \leq (\frac{4+\epsilon}{\epsilon})^{N\overline{d}\overline{r}^2 }$ where $\overline{r}=\max_i r_i$ and $\overline{d}=\max_i d_i$.

Without loss of the generality,  we assume that $\calX$ is in TT format with $\|\calX\|_F=1$.
For simplicity, we use $\calI$ to denote the index set $[n_1]\times \cdots \times [n_N]$. According to the construction of the $\epsilon$-net, there exists $p=(p_1,\dots, p_N)\in\calI$ such that
\begin{eqnarray}
    \label{L1 ProofOf<H,X>forSubGaussian_proof5}
    \|L(\mX_i) - L(\mX_i^{(p_i)}) \|\leq\epsilon, \ \  i\in[N-1] \ \ \  \text{and}  \ \ \  \|L(\mX_N) - L(\mX_N^{(p_N)})\|_2\leq\epsilon.
\end{eqnarray}
Taking $\epsilon=\frac{c\delta_{\overline{r}}}{N}$ with a positive constant $c$ gives
\begin{eqnarray}
    \label{L1 ProofOf<H,X>forSubGaussian_proof6}
    &\!\!\!\!\!\!\!\!&\sup_{\calX}\frac{1}{m}\|\calA(\calX -\calX^{(p)} )\|_1\nonumber\\
    &\!\!\!\!=\!\!\!\!&\sup_{[\mX_1,\dots, \mX_N]}\frac{1}{m}\sum_{k=1}^m|\<\calA_k, [\mX_1,\dots, \mX_N] - [\mX_1^{(p_1)},\dots, \mX_N^{(p_N)}]  \>|\nonumber\\
    &\!\!\!\!=\!\!\!\!&\sup_{[\mX_1,\dots, \mX_N]}\frac{1}{m}\sum_{k=1}^m|\<\calA_k, \sum_{a_1=1}^N[\mX_1^{(p_1)},\dots, \mX_{a_1}^{(p_{a_1})}-\mX_{a_1},  \dots, \mX_N]\>|\nonumber\\
    &\!\!\!\!\leq\!\!\!\!& \sup_{\calX}\frac{N \epsilon}{m}\|\calA(\calX )\|_1=  \sup_{\calX}\frac{c\delta_{\overline{r}}}{m}\|\calA(\calX )\|_1,
\end{eqnarray}
where we write $[\mX_1,\dots, \mX_N] - [\mX_1^{(p_1)},\dots, \mX_N^{(p_N)}]$ in the second line  as the sum of $N$ terms via \Cref{EXPANSION_A1TOAN-B1TOBN_1}.

According to \eqref{L1 ProofOf<H,X>forSubGaussian_proof6}, we have
\begin{eqnarray}
\label{Relationship of fixed and any tensor}
&\!\!\!\!\!\!\!\!&\sup_{\calX}\bigg|\frac{1}{m}\|\calA(\calX) \|_1 -\sqrt{2/\pi} \bigg|\nonumber\\
&\!\!\!\!\leq\!\!\!\!& \sup_{\calX^{(p)}}\bigg|\frac{1}{m}\|\calA(\calX^{(p)}) \|_1 -\sqrt{2/\pi} \bigg| + \sup_{\calX}\frac{1}{m}\|\calA(\calX -\calX^{(p)} )\|_1\nonumber\\
&\!\!\!\!\leq\!\!\!\!& \sup_{\calX^{(p)}}\bigg|\frac{1}{m}\|\calA(\calX^{(p)}) \|_1 -\sqrt{2/\pi} \bigg| + \sup_{\calX} \frac{c\delta_{\overline{r}}}{m}\|\calA(\calX) \|_1\nonumber\\
&\!\!\!\!\leq\!\!\!\!& \sup_{\calX^{(p)}}\bigg|\frac{1}{m}\|\calA(\calX^{(p)}) \|_1 -\sqrt{2/\pi} \bigg| + \sup_{\calX}c\delta_{\overline{r}} \bigg| \frac{1}{m}\|\calA(\calX) \|_1 -\sqrt{2/\pi}  \bigg| + c\delta_{\overline{r}}\sqrt{2/\pi},
\end{eqnarray}
and then it follows
\begin{eqnarray}
\label{Relationship of fixed and any tensor 1}
\sup_{\calX}\bigg|\frac{1}{m}\|\calA(\calX) \|_1 -\sqrt{2/\pi} \bigg|\leq\frac{\sup_{\calX^{(p)}}\bigg|\frac{1}{m}\|\calA(\calX^{(p)}) \|_1 -\sqrt{2/\pi} \bigg| + c\delta_{\overline{r}}\sqrt{2/\pi}}{1 - c\delta_{\overline{r}}}.
\end{eqnarray}

To finish our derivation, we need to obtain the concentration inequality with respect to $\frac{1}{m}\sum_{k=1}^m|\<\calA_k, \calX^{(p)} \>|$. For any fixed TT format tensor $\calX^{(p)}$
with $\|\calX^{(p)}\|_F = 1$, $|\<\calA_k, \calX^{(p)} \>|$ obeys standard Gaussian with mean $\sqrt{2/\pi}$ and unit variance since $\{\calA_k\}_{k=1}^m$  have i.i.d. standard Gaussian entries. Hence, based on the tail function of Gaussian random variable, we have
\begin{eqnarray}
\label{Tail function of gaussian random variable}
\P{\sup_{\calX^{(p)}}\bigg|\frac{1}{m}\sum_{k=1}^m|\<\calA_k, \calX^{(p)} \>| -\sqrt{2/\pi}  \bigg|\geq \frac{\delta_{\overline{r}}}{2}} \leq 2e^{-c_1m\delta_{\overline{r}}^2},
\end{eqnarray}
where $c_1$ is a constant.

Based on \eqref{Relationship of fixed and any tensor 1}, we can derive
\begin{eqnarray}
\label{Rip conclusion 1_1}
&\!\!\!\!\!\!\!\!&\P{ \sup_{\calX}\bigg|\frac{1}{m}\|\calA(\calX) \|_1 -\sqrt{2/\pi} \bigg|\geq \frac{\frac{\delta_{\overline{r}}}{2} + c\delta_{\overline{r}}\sqrt{2/\pi}}{1 - c\delta_{\overline{r}}}}\nonumber\\
&\!\!\!\!\leq\!\!\!\!& \P{ \frac{\sup_{\calX^{(p)}}\bigg|\frac{1}{m}\|\calA(\calX^{(p)}) \|_1 -\sqrt{2/\pi} \bigg| + c\delta_{\overline{r}}\sqrt{2/\pi}}{1 - c\delta_{\overline{r}}} \geq \frac{\frac{\delta_{\overline{r}}}{2} + c\delta_{\overline{r}}\sqrt{2/\pi}}{1 - c\delta_{\overline{r}}}}\nonumber\\
&\!\!\!\!\leq\!\!\!\!&\Pi_{i=1}^N n_i e^{1-c_1m\delta_{\overline{r}}^2}\nonumber\\
&\!\!\!\!\leq\!\!\!\!&\bigg(\frac{4+\epsilon}{\epsilon}\bigg)^{N\overline{d}\overline{r}^2 }e^{1-c_1m\delta_{\overline{r}}^2}\nonumber\\
&\!\!\!\!\leq\!\!\!\!&e^{1 - c_1m\delta_{\overline{r}}^2 + c_2 N\overline{d}\overline{r}^2 \log N},
\end{eqnarray}
where in the last line, we choose $\epsilon=\frac{c\delta_{\overline{r}}}{N}$, and $c_2$ is a positive constant. Based on \eqref{Rip conclusion 1_1}, we have
\begin{eqnarray}
\label{Rip conclusion 1}
\P{\sup_{\calX}\bigg|\frac{1}{m}\|\calA(\calX) \|_1 -\sqrt{2/\pi} \bigg|\leq \frac{\frac{\delta_{\overline{r}}}{2} + c\delta_{\overline{r}}\sqrt{2/\pi}}{1 - c\delta_{\overline{r}}}} \geq 1 - e^{1 - c_1m\delta_{\overline{r}}^2 + c_2 N\overline{d}\overline{r}^2 \log N}.
\end{eqnarray}

To guarantee that $\frac{\frac{\delta_{\overline{r}}}{2} + c\delta_{\overline{r}}\sqrt{2/\pi}}{1 - c\delta_{\overline{r}}}\leq \delta_{\overline{r}} \leq \sqrt{2/\pi}$, we can select $c=\frac{\sqrt{2\pi}}{8}$ in $\epsilon=\frac{c\delta_{\overline{r}}}{N}$. Furthermore, if $m\geq \Omega(N\overline{d}\overline{r}^2\log N/\delta_{\overline{r}}^2)$, we obtain the following result:
\begin{eqnarray}
\label{Rip conclusion 2}
\P{\sup_{\calX}\bigg|\frac{1}{m}\|\calA(\calX) \|_1 -\sqrt{2/\pi} \bigg| \leq \delta_{\overline{r}}} \geq 1 - e^{-c_3 N\overline{d}\overline{r}^2 \log N},
\end{eqnarray}
where $c_3$ is a constant. In other words, with probability at least $1 - e^{-c_3 N\overline{d}\overline{r}^2 \log N}$, it holds that
\begin{eqnarray}
    \label{Rip conclusion 3}
    (\sqrt{2/\pi}-\delta_{\overline{r}})\|\calX\|_F \leq \frac{1}{m}\|\calA(\calX)\|_1 \leq (\sqrt{2/\pi}+\delta_{\overline{r}})\|\calX\|_F
\end{eqnarray}
for any low-TT-rank tensor $\calX\in\R^{d_{1} \times \cdots \times d_{N}}$ with ranks $(r_1,\dots, r_{N-1})$.

\end{proof}

\section{Proof of \Cref{sharpness property lemma}}
\label{Proof of L1 RIP property in appendix}

\begin{proof}
We first expand $\frac{1}{m}\|\calA(\calX - \calX^\star) -\vs \|_1 - \frac{1}{m}\|\vs\|_1$ as
\begin{eqnarray}
    \label{expansion of two terms with ourlier}
    &\!\!\!\!\!\!\!\!&\frac{1}{m}\|\calA(\calX - \calX^\star) -\vs \|_1 - \frac{1}{m}\|\vs\|_1\nonumber\\
    &\!\!\!\!=\!\!\!\!& \frac{1}{m}\|\calA_{\calS^c}(\calX - \calX^\star) \|_1 + \frac{1}{m}\|\calA_{\calS}(\calX - \calX^\star) -\vs \|_1 - \frac{1}{m}\|\vs\|_1\nonumber\\
    &\!\!\!\!\geq \!\!\!\!&\frac{1}{m}\|\calA_{\calS^c}(\calX - \calX^\star) \|_1 - \frac{1}{m}\|\calA_{\calS}(\calX - \calX^\star) \|_1.
\end{eqnarray}
In the subsequent part, we focus on analyzing the lower bound of $\frac{1}{m}\|\calA_{\calS^c}(\calX) \|_1 - \frac{1}{m}\|\calA_{\calS}(\calX) \|_1$ for any tensor $\calX$ in TT format with TT ranks smaller than $2\ol r$. We can construct an $\epsilon$-net $\{ \calX^{(p)} \}$ with the covering number $(\frac{4+\epsilon}{\epsilon})^{4N\overline{d}\overline{r}^2 }$ for any low-TT-rank tensor $\calX$ with ranks $(2r_1,\dots, 2r_{N-1})$ such that \eqref{L1 ProofOf<H,X>forSubGaussian_proof6} holds.  Without loss of the generality, we assume that $\calX = [\mX_1,\dots, \mX_N]$ is in  left-orthogonal TT format with $\|\calX\|_F=1$. Then we define
\begin{eqnarray}
    \label{Rip property set}
    Y_k  =\begin{cases}
    - |\<\calA_k,  \calX^{(p)}\>| + \sqrt{2/\pi}, & k\in \calS,\\
    |\<\calA_k,  \calX^{(p)}\>| - \sqrt{2/\pi}, & k\in \calS^c.
  \end{cases}
\end{eqnarray}
Hoeffding inequality for Gaussian random variables tells that
\begin{eqnarray}
    \label{Rip property concentration inequality}
    \P{\frac{1}{m}\sum_{k=1}^m Y_k = \frac{1}{m}\|\calA_{\calS^c} (\calX^{(p)})\|_1 - \frac{1}{m}\|\calA_{\calS} (\calX^{(p)})\|_1 - (1-2p_s)\sqrt{2/\pi} \geq -\frac{\delta_{2\overline{r}}}{2} } \geq 1 - e^{-c_1 m \delta_{2\overline{r}}^2},
\end{eqnarray}
where $c_1$ is a constant.

On the other hand, $
    \frac{1}{m}\|\calA_{\calS^c}(\calX)\|_1 \geq  \frac{1}{m} \| \calA_{\calS^c}(\calX^{(p)})\|_1 - \frac{1}{m}\|\calA_{\calS^c}(\calX - \calX^{(p)})\|_1$ and $\frac{1}{m}\|\calA_{\calS}(\calX)\|_1 \leq  \frac{1}{m} \| \calA_{\calS}(\calX^{(p)})\|_1 + \frac{1}{m}\|\calA_{\calS}(\calX - \calX^{(p)})\|_1$ hold for any tensor $\calX$. Hence, we have
\begin{eqnarray}
    \label{Rip property norm property1}
    \frac{1}{m}\|\calA_{\calS^c}(\calX)\|_1 - \frac{1}{m}\|\calA_{\calS}(\calX)\|_1 \geq \frac{1}{m} \| \calA_{\calS^c}(\calX^{(p)})\|_1 - \frac{1}{m} \| \calA_{\calS}(\calX^{(p)})\|_1 - \frac{1}{m}\|\calA(\calX - \calX^{(p)}) \|_1.
\end{eqnarray}

Combing \eqref{L1 ProofOf<H,X>forSubGaussian_proof6}, we can get
\begin{eqnarray}
    \label{Rip property norm property2}
    &\!\!\!\!\!\!\!\!&\inf_{\calX}\bigg(\frac{1}{m}\|\calA_{\calS^c}(\calX)\|_1 - \frac{1}{m}\|\calA_{\calS}(\calX)\|_1\bigg) \nonumber\\
    &\!\!\!\!\geq\!\!\!\!& \inf_{\calX^{(p)}}\bigg(\frac{1}{m} \| \calA_{\calS^c}(\calX^{(p)})\|_1 - \frac{1}{m} \| \calA_{\calS}(\calX^{(p)})\|_1 \bigg) - \sup_{\calX}\frac{1}{m}\|\calA(\calX - \calX^{(p)}) \|_1\nonumber\\
    &\!\!\!\!\geq\!\!\!\!& \inf_{\calX^{(p)}}\bigg(\frac{1}{m} \| \calA_{\calS^c}(\calX^{(p)})\|_1 - \frac{1}{m} \| \calA_{\calS}(\calX^{(p)})\|_1 \bigg) - \sup_{\calX}\frac{c\delta_{2\overline{r}}}{m} \|\calA(\calX)\|_1\nonumber\\
    &\!\!\!\!\geq\!\!\!\!& \inf_{\calX^{(p)}}\bigg(\frac{1}{m} \| \calA_{\calS^c}(\calX^{(p)})\|_1 - \frac{1}{m} \| \calA_{\calS}(\calX^{(p)})\|_1 \bigg)  - c\delta_{2\overline{r}}\bigg(\frac{\delta_{2\overline{r}}/2 + c\delta_{2\overline{r}}\sqrt{2/\pi}}{1 - c\delta_{2\overline{r}}} + \sqrt{2/\pi}\bigg),
\end{eqnarray}
where the last line follows the \eqref{Rip conclusion 1} with probability $1 - e^{-c_2 N\overline{d}\overline{r}^2 \log N}$ in which $c_2$ is a positive constant.
Denote the event $F:=\{ \text{\eqref{Rip conclusion 1} is satisfied} \}$ which holds with probability $1 - e^{-c_2 N\overline{d}\overline{r}^2 \log N}$.
We can take the union bound with \eqref{Rip property concentration inequality} to conclude
\begin{eqnarray}
    \label{Rip property conclusion1}
    &\!\!\!\!\!\!\!\!&\P{ \inf_{\calX}\bigg(\frac{1}{m}\|\calA_{\calS^c}(\calX)\|_1 - \frac{1}{m}\|\calA_{\calS}(\calX)\|_1\bigg) \geq (1-2p_s)\sqrt{2/\pi} - \frac{\delta_{2\overline{r}}}{2} - c\delta_{2\overline{r}}\bigg(\frac{\delta_{2\overline{r}}/2 + c\delta_{2\overline{r}}\sqrt{2/\pi}}{1 - c\delta_{2\overline{r}}} + \sqrt{2/\pi}\bigg)\bigg| F }\nonumber\\
    &\!\!\!\!\geq\!\!\!\!& 1 - \Pi_{i=1}^N n_i 2e^{c_1m\delta_{2\overline{r}}^2}\nonumber\\
    &\!\!\!\!\geq \!\!\!\!&1 - (\frac{4+\epsilon}{\epsilon})^{4N\overline{d}\overline{r}^2 }e^{1-c_1m\delta_{2\overline{r}}^2}\nonumber\\
    &\!\!\!\!\geq\!\!\!\!& 1 - e^{1 - c_1m\delta_{2\overline{r}}^2 + c_3 N\overline{d}\overline{r}^2 \log N},
\end{eqnarray}
where $c_3$ is a positive constant.

To guarantee that $- \frac{\delta_{2\overline{r}}}{2} - c\delta_{2\overline{r}}(\frac{\delta_{2\overline{r}}/2 + c\delta_{2\overline{r}}\sqrt{2/\pi}}{1 - c\delta_{2\overline{r}}} + \sqrt{2/\pi}) \geq -\delta_{2\overline{r}}$ and $\delta_{2\overline{r}}\leq \sqrt{2/\pi}$, we set $c=\frac{\sqrt{2\pi}}{8}$.
While $m\geq \Omega(N\overline{d}\overline{r}^2\log N/\delta_{2\overline{r}}^2)$, we can obtain
\begin{eqnarray}
\label{Rip property conclusion2}
\P{\inf_{\calX}\big(\frac{1}{m}\|\calA_{\calS^c}(\calX)\|_1 - \frac{1}{m}\|\calA_{\calS}(\calX)\|_1\big) \geq (1-2p_s)\sqrt{2/\pi} - \delta_{2\overline{r}} \bigg| F} \geq 1 - e^{-c_4 N\overline{d}\overline{r}^2 \log N},
\end{eqnarray}
where $c_4$ is a constant.

In the end, we can derive
\begin{eqnarray}
\label{Rip property conclusion3}
&\!\!\!\!\!\!\!\!&\P{\inf_{\calX}\big(\frac{1}{m}\|\calA_{\calS^c}(\calX)\|_1 - \frac{1}{m}\|\calA_{\calS}(\calX)\|_1\big) \geq (1-2p_s)\sqrt{2/\pi} - \delta_{2\overline{r}}}\nonumber\\
&\!\!\!\!\geq\!\!\!\!&\P{\inf_{\calX}\big(\frac{1}{m}\|\calA_{\calS^c}(\calX)\|_1 - \frac{1}{m}\|\calA_{\calS}(\calX)\|_1\big) \geq (1-2p_s)\sqrt{2/\pi} - \delta_{2\overline{r}} \cap F}\nonumber\\
&\!\!\!\!=\!\!\!\!&\P{F}\P{\inf_{\calX}\big(\frac{1}{m}\|\calA_{\calS^c}(\calX)\|_1 - \frac{1}{m}\|\calA_{\calS}(\calX)\|_1\big) \geq (1-2p_s)\sqrt{2/\pi} - \delta_{2\overline{r}} \bigg| F}\nonumber\\
&\!\!\!\!\geq\!\!\!\!&(1 - e^{-c_2 N\overline{d}\overline{r}^2 \log N})(1 - e^{-c_4 N\overline{d}\overline{r}^2 \log N}) \geq 1 - 2 e^{-c_5 N\overline{d}\overline{r}^2 \log N},
\end{eqnarray}
where $c_5$ is a positive constant.
\end{proof}

\section{Proof of \Cref{Robust Regularity condition of full tensor}}
\label{Proof of the Robust regularity condition for full tensor}

\begin{proof}
We can apply the norm inequalities to derive
\begin{eqnarray}
    \label{The lower bound of cross term in the robust regularity condition}
    \< \partial f(\calX), \calX - \calX^\star \> & \!\!\!\!= \!\!\!\!&\frac{1}{m}\sum_{k=1}^m \text{sign}(\<\calA_k,\calX\> - y_k)(\<\calA_k, \calX - \calX^\star \> - s_k)+\frac{1}{m}\sum_{k=1}^m \text{sign}(\<\calA_k,\calX\> - y_k)s_k\nonumber\\
    &\!\!\!\!\geq\!\!\!\!&\frac{1}{m}\|\calA(\calX - \calX^\star) -\vs \|_1 - \frac{1}{m}\|\vs\|_1\nonumber\\
    &\!\!\!\!\geq\!\!\!\!& ((1-2p_s)\sqrt{2/\pi} - \delta_{2\overline{r}})\|\calX - \calX^\star\|_F,
\end{eqnarray}
where the last line follows \Cref{sharpness property lemma} for $\delta_{2\overline{r}} \leq (1-2p_s)\sqrt{2/\pi}$.

\end{proof}

\section{Proof of \Cref{Local Convergence of PGD in the sensing_Theorem}}
\label{Proof of local convergence of PGD in appendix}

\begin{proof}
We expand $\|\calX^{(t+1)} - \calX^\star\|_F^2$ as following:
\begin{eqnarray}
    \label{The MSE relationship of different time}
    &\!\!\!\!\!\!\!\!&\|\calX^{(t+1)} - \calX^\star\|_F^2= \|\text{SVD}_{\vr}^{tt}(\calX^{(t)} - \mu_t\partial f(\calX^{(t)}))  - \calX^\star  \|_F^2\nonumber\\
    &\!\!\!\!\leq\!\!\!\!& \bigg(1+ \frac{600N}{\underline{\sigma}({\calX^\star})}\|\calX^{(0)} - \calX^\star \|_F \bigg)\| \calX^{(t)} - \mu_t\partial f(\calX^{(t)})   - \calX^\star \|_F^2\nonumber\\
    &\!\!\!\!=\!\!\!\!&\bigg(1+ \frac{600N}{\underline{\sigma}({\calX^\star})}\|\calX^{(0)} - \calX^\star \|_F \bigg)\big(\|\calX^{(t)} - \calX^\star\|_F^2 -2\mu_t\< \partial f(\calX^{(t)}), \calX^{(t)} - \calX^\star \> + \mu_t^2\|\partial f(\calX^{(t)})\|_F^2\big)\nonumber\\
    &\!\!\!\!\leq\!\!\!\!& \bigg(1+ \frac{600N}{\underline{\sigma}({\calX^\star})}\|\calX^{(0)} - \calX^\star \|_F \bigg) \big(\|\calX^{(t)} - \calX^\star\|_F^2 -2\mu_t((1-2p_s)\sqrt{2/\pi} - \delta_{2\overline{r}})\|\calX^{(t)} - \calX^\star\|_F\nonumber\\
    &\!\!\!\!\!\!\!\!& + \mu_t^2(\sqrt{2/\pi} + \delta_{2\overline{r}})^2\big),
\end{eqnarray}
where we use \Cref{Perturbation bound for TT SVD} in the first inequality and subsequently employ \Cref{Robust Regularity condition of full tensor} and \Cref{Upper bound of subgradient of f} in the last line.

Under the initial condition $\|\calX^{(0)} - \calX^\star\|_F\leq \frac{c\underline{\sigma}({\calX^\star})}{600N}$ with a constant $c$, \eqref{The MSE relationship of different time} can be rewritten as
\begin{eqnarray}
    \label{The expansion of difference two tensors final conclusion1}
    \| \calX^{(t+1)}   - \calX^\star \|_F^2 \leq (1+c)\big( \|\calX^{(t)} - \calX^\star\|_F^2 -2\mu_t((1-2p_s)\sqrt{2/\pi} - \delta_{2\overline{r}})\|\calX^{(t)} - \calX^\star\|_F + \mu_t^2(\sqrt{2/\pi} + \delta_{2\overline{r}})^2 \big).
\end{eqnarray}

Based on the discussion in {Appendix} \ref{Proof of the induction in the full tensor}, we can select $\mu_t=\lambda q^t$ where $\lambda = \frac{ ((1-2p_s)\sqrt{2/\pi} - \delta_{2\overline{r}})}{2(\sqrt{2/\pi} + \delta_{2\overline{r}})^2}\|\calX^{(0)} - \calX^\star\|_F$ and $q = \sqrt{(1+c) (1-\frac{3((1-2p_s)\sqrt{2/\pi} - \delta_{2\overline{r}})^2}{4(\sqrt{2/\pi} + \delta_{2\overline{r}})^2})}$, and then obtain
\begin{eqnarray}
    \label{The expansion of difference two tensors final conclusion2}
    \| \calX^{(t)}   - \calX^\star \|_F^2 \leq \|\calX^{(0)} - \calX^\star\|_F^2 q^{2t}.
\end{eqnarray}

\end{proof}

\subsection{Proof of \eqref{The expansion of difference two tensors final conclusion2}}
\label{Proof of the induction in the full tensor}
\begin{proof}
Define
\begin{eqnarray}
    \label{the definitions of constants in the induction of full tensor1}
    &&c_1 = 2((1-2p_s)\sqrt{2/\pi} - \delta_{2\overline{r}}),\\
    \label{the definitions of constants in the induction of full tensor2}
    &&c_2 = (\sqrt{2/\pi} + \delta_{2\overline{r}})^2.
\end{eqnarray}

Now we can simplify \eqref{The expansion of difference two tensors final conclusion1} as
\begin{eqnarray}
    \label{The expansion of difference two tensors final conclusion1 simplification}
    \| \calX^{(t+1)}   - \calX^\star \|_F^2 \leq (1+c)( \|\calX^{(t)} - \calX^\star\|_F^2 -\mu_t c_1\|\calX^{(t)} - \calX^\star\|_F + \mu_t^2c_2).
\end{eqnarray}

Next, we aim to show
\begin{eqnarray}
    \label{The goal of proof in the induction of full tensor}
    \|\calX^{(t)} - \calX^\star\|_F \leq c_0q^{\frac{t}{2}}  =  c_0(1+c)^{\frac{t}{2}}p^{\frac{t}{2}}
\end{eqnarray}
where $c_0 = \|\calX^{(0)} - \calX^\star\|_F$ and $p$ is a parameter that needs to be determined. Let us therefore fix a value $x\in[0,1]$ satisfying $\|\calX^{(t)} - \calX^\star\|_F = x c_0(1+c)^{\frac{t}{2}}p^{\frac{t}{2}}$. Assume the above induction hypothesis \eqref{The goal of proof in the induction of full tensor} holds at the $t$-iteration. We need to further prove
\begin{eqnarray}
    \label{The goal of proof in the induction of full tensor t+1}
    \|\calX^{(t+1)} - \calX^\star\|_F^2 \leq  (1+c)\big( c_0^2(1+c)^{t}p^{t} x^2  -   \mu_t c_0 c_1(1+c)^{\frac{t}{2}}p^{\frac{t}{2}} x  +\mu_t^2c_2   \big) \leq c_0^2(1+c)^{t+1}p^{t+1}.
\end{eqnarray}
When we select $\mu_t = \lambda q^{t} = c_0 a (1+c)^{\frac{t}{2}}p^{\frac{t}{2}}$ where $a$ is a parameter which needs to be determined, \eqref{The goal of proof in the induction of full tensor t+1} can be simplified as
\begin{eqnarray}
    \label{The goal of proof in the induction of full tensor t+1 simplified}
    x^2  - c_1 a x + c_2a^2 \leq p.
\end{eqnarray}
Note that the left hand side of \eqref{The goal of proof in the induction of full tensor t+1 simplified} is a convex quadratic in $x$ and therefore the maximum between $[0,1]$ must occur either at $x=0$ or $x=1$.
\begin{itemize}
  \item{When selecting $x=0$, we can derive the condition $c_2a^2 \leq p$.}
  \item{When choosing $x=1$, we have $c_2a^2 - c_1 a + 1-p \leq 0$. Since $c_1,c_2>0$, we have $\frac{c_1 - \sqrt{c_1^2-4c_2(1-p)}}{2c_2} \leq a \leq \frac{c_1 + \sqrt{c_1^2-4c_2(1-p)}}{2c_2}$.}
\end{itemize}
By selecting $p = 1-\frac{3c_1^2}{16c_2}$ and $a = \frac{c_1}{4c_2}$, we can ensure that conditions $c_2a^2 \leq p$ and $c_2a^2 - c_1 a + 1-p \leq 0$ holds for $\delta_{2\ol r}\geq 0$. Hence we can respectively choose $\lambda = \frac{c_0c_1}{4c_2}$ and $q = \sqrt{(1+c) (1-\frac{3c_1^2}{16c_2})}$, which further guarantees \eqref{The goal of proof in the induction of full tensor t+1}.

\end{proof}

\section{Proof of \Cref{Robust Regularity condition of factor tensor}}
\label{Proof of the Robust regularity condition for factor tensor}
\begin{proof}
First, we provide one useful property. According to
\begin{eqnarray}
    \label{upper bound of the distance}
    \text{dist}^2(\{\mX_i\},\{ \mX_i^\star\})\leq \frac{\underline{\sigma}^2(\calX^\star) ((1-2p_s)\sqrt{2/\pi} -  \delta_{(N+1)\overline{r}})^2 }{18(2N^2 - 2N +1)(N+1+\sum_{i=2}^{N-1}r_i)(\sqrt{2/\pi} + \delta_{(N+1)\overline{r}})^2}
\end{eqnarray}
which can be obtained by $\calX\in\calC(b)$ and \Cref{LOWER BOUND OF TWO DISTANCES}, we can obtain
\begin{eqnarray}
\label{upper bound TT spctral norm robust}
    \sigma_1^2({\calX}^{\<i \>}) &\!\!\!\!= \!\!\!\!&  \|{\calX}^{\geq i+1} \|^2 \leq \min_{\mR_i\in\O^{r_i\times r_i}}2\|\mR_{i}^\top{\calX^{\star}}^{\geq i+1} \|^2 + 2\|{\calX}^{\geq i+1} - \mR_{i}^\top{\calX^{\star}}^{\geq i+1} \|^2\nonumber\\
    &\!\!\!\!\leq \!\!\!\!&2\ol{\sigma}^2(\calX^\star) + \min_{\mR_i\in\O^{r_i\times r_i}}2 \|{\calX}^{\<i \>} - {\calX^\star}^{\<i \>} + {\calX^\star}^{\<i \>} - {\calX}^{\leq i}\mR_{i}^\top{\calX^{\star}}^{\geq i+1}  \|^2\nonumber\\
    &\!\!\!\!\leq \!\!\!\!&  2\ol{\sigma}^2(\calX^\star) + 4\|\calX - \calX^\star \|_F^2 + \min_{\mR_i\in\O^{r_i\times r_i}}4 \|\mR_{i}^\top{\calX^{\star}}^{\geq i+1} \|^2 \|{\calX}^{\leq i} - {\calX^\star}^{\leq i} \mR_i \|_F^2\nonumber\\
    &\!\!\!\!\leq \!\!\!\!& 2\ol{\sigma}^2(\calX^\star) + \bigg(4 + \frac{16\ol{\sigma}^2(\calX^\star)}{\underline{\sigma}^2(\calX^\star)}\bigg)\|\calX-\calX^\star\|_F^2\nonumber\\
    &\!\!\!\!\leq \!\!\!\!& 2\ol{\sigma}^2(\calX^\star) + \frac{45N\ol{\sigma}^2(\calX^\star)}{\underline{\sigma}^2(\calX^\star)} \text{dist}^2(\{\mX_i \},\{ \mX_i^\star \})\leq \frac{9\ol{\sigma}^2(\calX^\star)}{4}, i\in[N-1],
\end{eqnarray}
where the fourth and last lines respectively follow \cite[eq. (60)]{qin2024guaranteed} and \Cref{LOWER BOUND OF TWO DISTANCES}.  Note that  $\ol{\sigma}^2(\calX) = \max_{i=1}^{N-1}\sigma_1^2({\calX}^{\<i \>})\leq \frac{9\ol{\sigma}^2(\calX^\star)}{4}$.

Then we need to define the subgradient of $F(\mX_1, \dots, \mX_N)$ as following:
\begin{eqnarray}
    \label{The subgradient of the function for factor tensor}
    \partial_{L(\mX_{i})} F(\mX_1, \dots, \mX_N) = \begin{bmatrix}\partial_{\mX_{i}(1)} F(\mX_1, \dots, \mX_N)\\ \vdots \\ \partial_{\mX_{i}(d_i)} F(\mX_1, \dots, \mX_N)  \end{bmatrix}.
\end{eqnarray}
Here the subgradient with respect to each factor $\mX_{i}(s_i)$ can be computed as
\begin{align*}
\partial_{\mX_{i}(s_{i})}F(\mX_1, \dots, \mX_N)=\frac{1}{m}\sum_{k=1}^{m} \text{sign}(\<\calA_k,\calX\>-y_k)\sum_{s_1,\ldots,s_{i-1},\atop s_{i+1},\ldots,s_N } \Big(& \calA_k(s_1,\dots,s_N)\mX_{i-1}^\top(s_{i-1})\cdots\mX_{1}^\top(s_{1})\cdot\\
& \mX_{N}^\top(s_{N})\cdots\mX_{i+1}^\top(s_{i+1}) \Big).
\end{align*}

Before analyzing the robust regularity condition, we need to define three matrices for $i\in[N]$ as follows:
\begin{eqnarray}
    \label{The definition of D1 robust L1 loss}
    \mD_1(i) &\!\!\!\!=\!\!\!\!& \begin{bmatrix} \mX_{i-1}^\top(1)\!\cdots\!\mX_{1}^\top(1) \ \ \ \   \cdots \ \ \ \   \mX_{i-1}^\top(d_{i-1})\!\cdots\!\mX_{1}^\top(d_{1})   \end{bmatrix}\nonumber\\
    &\!\!\!\!=\!\!\!\!&L^\top(\mX_{i-1})\ol \otimes \cdots \ol \otimes L^\top(\mX_1)\in\R^{r_i\times(d_1\cdots d_{i-1})},\\
    \mD_2(i) &\!\!\!\!=\!\!\!\!& \begin{bmatrix} \mX_{N}^\top(1)\cdots \mX_{i+1}^\top(1)\\ \vdots \\ \mX_{N}^\top(d_N)\cdots \mX_{i+1}^\top(d_{i+1}) \end{bmatrix}\in\R^{(d_{i+1}\cdots d_N )\times r_i },
\end{eqnarray}
where we note that $\mD_1(1) = 1$ and $\mD_2(N)=1$.
Moreover, for each $s_i\in [d_i]$, we define matrix $\mE(s_i)\in\R^{(d_1\cdots d_{i-1})\times (d_{i+1}\cdots d_N)}$ whose $(s_1\cdots s_{i-1}, s_{i+1}\cdots s_N)$-th element  is given by
\begin{eqnarray}
    \label{Each element of E_si robust L1 loss}
    \mE(s_i)(s_1\cdots s_{i-1}, s_{i+1}\cdots s_N) =  \frac{1}{m}\sum_{k=1}^{m} \text{sign}(\<\calA_k,\calX\>-y_k)\calA_k(s_1,\dots,s_N).
\end{eqnarray}

Based on the aforementioned notations, we can derive
\begin{eqnarray}
    \label{PROJECTED GRADIENT DESCENT SQUARED TERM 1 to N robust L1 loss}
    \big\| \partial_{L(\mX_{i})}F(\mX_1, \dots, \mX_N)\big\|_F^2&\!\!\!\!=\!\!\!\!& \sum_{s_i=1}^{d_i}\big\| \partial_{\mX_{i}(s_i)} F(\mX_1, \dots, \mX_N)\big\|_F^2 =\sum_{s_i=1}^{d_i}\|\mD_1(i) \mE(s_i) \mD_2(i)   \|_F^2\nonumber\\
    &\!\!\!\!\leq\!\!\!\!& \sum_{s_i=1}^{d_i}\|L^\top(\mX_{i-1})\ol \otimes \cdots \ol \otimes L^\top(\mX_1)\|^2\| \mD_2(i)\|^2 \|\mE(s_i)\|_F^2\nonumber\\
    &\!\!\!\!\leq\!\!\!\!&\|L(\mX_1)\|^2\cdots\|L(\mX_{i-1})\|^2 \|{\calX}^{\geq i+1} \|^2 \|\frac{1}{m}\sum_{k=1}^{m}\text{sign}(\<\calA_{k},\calX\> - y_k)\calA_k\|_F^2\nonumber\\
    &\!\!\!\!\leq\!\!\!\!&\begin{cases}
    \frac{9\ol{\sigma}^2(\calX^\star)}{4}(\sqrt{2/\pi} +\delta_{2\overline{r}})^2, & i\in[N-1],\\
    (\sqrt{2/\pi} +\delta_{2\overline{r}})^2, & i = N,
  \end{cases}
\end{eqnarray}
where we use \eqref{KRONECKER PRODUCT VECTORIZATION11 - 2}, $\|\mD_2(i)\| =  \|{\calX}^{\geq i+1} \| $ and $\sum_{s_i=1}^{d_i}\|\mE(s_i)\|_F^2 = \|\frac{1}{m}\sum_{k=1}^{m}\text{sign}(\<\calA_{k},\calX\> - y_k)\calA_k\|_F^2$ in the second inequality. In addition, the third inequality follows $\|{\calX}^{\geq i+1} \| = \sigma_1({\calX}^{\<i \>})\leq \frac{3 \ol{\sigma}(\calX^\star)}{2}$ and \Cref{Upper bound of subgradient of f}.

Now, we rewrite the cross term in the robust regularity condition as following:
\begin{eqnarray}
    \label{Lower BOUND OF cross TERM OF Riemannian GD IN L1 SENSING Conclusion original}
    &\!\!\!\!\!\!\!\!&\sum_{i=1}^{N} \bigg\< L(\mX_i)-L_{\mR}(\mX_i^\star),\mathcal{P}_{\text{T}_{L(\mX_i)} \text{St}}\big(\partial_{L(\mX_{i})}F(\mX_1, \dots, \mX_N)\big)\bigg\>\nonumber\\
    &\!\!\!\!=\!\!\!\!&\sum_{i=1}^{N} \bigg\< L(\mX_i)-L_{\mR}(\mX_i^\star),\partial_{L(\mX_{i})}F(\mX_1, \dots, \mX_N)\bigg\> - T\nonumber\\
    &\!\!\!\!=\!\!\!\!&\frac{1}{m}\sum_{k=1}^m \text{sign}(\<\calA_k,\calX\> - y_k)\<\calA_k, \calX - \calX^\star \> + \frac{1}{m}\sum_{k=1}^m \text{sign}(\<\calA_k,\calX\> - y_k)\<\text{vec}(\calA_k), \vh  \> - T,
\end{eqnarray}
where
\begin{eqnarray}
    \label{H_T IN THE CROSS TERM}
    \vh&\!\!\!\!=\!\!\!\!&L_{\mR}(\mX_1^\star)\ol \otimes  \cdots \ol \otimes L_{\mR}({\mX}_N^\star) - L(\mX_1) \ol \otimes \cdots \ol \otimes L(\mX_{N-1}) \ol \otimes L_{\mR}(\mX_{N}^\star)\nonumber\\
    &\!\!\!\!\!\!\!\!& +\sum_{i=1}^{N-1}L(\mX_1)\ol \otimes\cdots\ol \otimes L(\mX_{i-1})\ol \otimes ( L({\mX}_i) -L_{\mR}({\mX}_i^\star) ) \ol \otimes L(\mX_{i+1})\ol \otimes \cdots \ol \otimes L(\mX_N),\\
\label{PROJECTION ORTHOGONAL IN Stiefel UPPER BOUND in the L1 sensing orignal}
    T& \!\!\!\!=\!\!\!\! &\sum_{i=1}^{N-1}\bigg\<L(\mX_i)-L_{\mR}(\mX_i^\star), \calP^{\perp}_{\text{T}_{L(\mX_i)} \text{St}}(\partial_{L(\mX_{i})}F(\mX_1, \dots, \mX_N)) \bigg\>.
\end{eqnarray}

To get the lower bound of \eqref{Lower BOUND OF cross TERM OF Riemannian GD IN L1 SENSING Conclusion original}, we need to obtain upper bounds of \eqref{H_T IN THE CROSS TERM} and \eqref{PROJECTION ORTHOGONAL IN Stiefel UPPER BOUND in the L1 sensing orignal}. According to \cite[eq. (82)]{qin2024guaranteed}, we directly obtain
\begin{eqnarray}
    \label{H_T IN THE CROSS TERM UPPER BOUND}
    \|\vh\|_2^2 \leq \frac{9N(N-1)}{8\ol{\sigma}^2(\calX^\star)}\text{dist}^4(\{\mX_i\},\{ \mX_i^\star\}).
\end{eqnarray}
Then, we can derive
\begin{eqnarray}
\label{PROJECTION ORTHOGONAL IN Stiefel UPPER BOUND in the L1 sensing}
    T&\!\!\!\!=\!\!\!\!&\sum_{i=1}^{N-1}\bigg\<\mathcal{P}^{\perp}_{\text{T}_{L(\mX_i)} \text{St}}(L(\mX_i)-L_{\mR}(\mX_i^\star)), \partial_{L(\mX_{i})}F(\mX_1, \dots, \mX_N) \bigg\>\nonumber\\
    &\!\!\!\!\leq\!\!\!\!&\frac{1}{2}\sum_{i=1}^{N-1}\|L(\mX_i)\|\|L(\mX_i)-L_{\mR}(\mX_i^\star)\|_F^2\big\|\partial_{L(\mX_{i})}F(\mX_1, \dots, \mX_N)\big\|_F\nonumber\\
    &\!\!\!\!\leq\!\!\!\!&\frac{3\ol{\sigma}(\calX^\star)}{4}(\sqrt{2/\pi} +\delta_{2\overline{r}})\sum_{i=1}^{N-1}\|L(\mX_i)-L_{\mR}(\mX_i^\star)\|_F^2\nonumber\\
    &\!\!\!\!\leq\!\!\!\!&\frac{3}{4\ol{\sigma}(\calX^\star)}(\sqrt{2/\pi} +\delta_{2\overline{r}})\text{dist}^2(\{\mX_i\},\{ \mX_i^\star\}),
\end{eqnarray}
where the first inequality follows \eqref{PROJECTED GRADIENT DESCENT SQUARED TERM 1 to N robust L1 loss} and $\mathcal{P}^{\perp}_{\text{T}_{L(\mX_i)} \text{St}}(\cdot)$ is defined as
\begin{eqnarray}
\label{Projection orthogonal in the Stiefel gradient descent L1}
    &\!\!\!\!\!\!\!\!&\calP^{\perp}_{\text{T}_{L(\mX_i)} \text{St}}(L(\mX_i)-L_{\mR}(\mX_i^\star))\nonumber\\
    &\!\!\!\!=\!\!\!\!& L(\mX_i)-L_{\mR}(\mX_i^\star) - \calP_{\text{T}_{L(\mX_i)} \text{St}}(L(\mX_i)-L_{\mR}(\mX_i^\star))\nonumber\\
    &\!\!\!\!=\!\!\!\!&\frac{1}{2}L(\mX_i)((L(\mX_i)-L_{\mR}(\mX_i^\star))^\top L(\mX_i)+L^\top(\mX_i)(L(\mX_i)-L_{\mR}(\mX_i^\star)))\nonumber\\
    &\!\!\!\!=\!\!\!\!&\frac{1}{2}L(\mX_i)(2\mId_{r_i}-L_{\mR}^\top(\mX_i^\star)L(\mX_i)-L^\top(\mX_i)L_{\mR}(\mX_i^\star))\nonumber\\
    &\!\!\!\!=\!\!\!\!&\frac{1}{2}L(\mX_i)((L(\mX_i)-L_{\mR}(\mX_i^\star))^\top(L(\mX_i)-L_{\mR}(\mX_i^\star))).
\end{eqnarray}
Ultimately, we arrive at
\begin{eqnarray}
    \label{Lower BOUND OF cross TERM OF Riemannian GD IN L1 SENSING Conclusion}
    &\!\!\!\!\!\!\!\!&\sum_{i=1}^{N} \bigg\< L(\mX_i)-L_{\mR}(\mX_i^\star),\mathcal{P}_{\text{T}_{L(\mX_i)} \text{St}}\big(\partial_{L(\mX_{i})}F(\mX_1, \dots, \mX_N)\big)\bigg\>\nonumber\\
    &\!\!\!\!\geq\!\!\!\!& \frac{1}{m}\sum_{k=1}^m | \<\calA_k,\calX- \calX^\star\> - s_k | - \frac{1}{m}\sum_{k=1}^m |s_k| - \frac{1}{m}\sum_{k=1}^m|\<\text{vec}(\calA_k), \vh  \>| - \frac{3}{4\ol{\sigma}(\calX^\star)}(\sqrt{2/\pi} +\delta_{2\overline{r}})\text{dist}^2(\{\mX_i\},\{ \mX_i^\star\})\nonumber\\
    &\!\!\!\!\geq\!\!\!\!& ((1-2p_s)\sqrt{2/\pi} - \delta_{2\overline{r}})\|\calX - \calX^\star\|_F - (\sqrt{2/\pi} + \delta_{(N+1)\overline{r}})\|\vh\|_2  - \frac{3}{4\ol{\sigma}(\calX^\star)}(\sqrt{2/\pi} +\delta_{2\overline{r}})\text{dist}^2(\{\mX_i\},\{ \mX_i^\star\})\nonumber\\
    &\!\!\!\!\geq\!\!\!\!&\frac{\underline{\sigma}(\calX^\star)((1-2p_s)\sqrt{2/\pi} - \delta_{(N+1)\overline{r}})}{4\sqrt{2(N+1+\sum_{i=2}^{N-1}r_i)}\ol{\sigma}(\calX^\star)}\text{dist}(\{\mX_i\},\{ \mX_i^\star\})
\end{eqnarray}
where \eqref{PROJECTION ORTHOGONAL IN Stiefel UPPER BOUND in the L1 sensing} is used in the first inequality. Note that $\vh$ can be viewed as a TT format where the rank is at most $((N-1)r_1,\dots,(N-1)r_{N-1})$. Hence, we apply $\ell_1/\ell_2$-RIP and \Cref{sharpness property lemma} in the second inequality. The last line follows $\delta_{2\overline{r}}\leq\delta_{(N+1)\overline{r}}\leq (1-2p_s)\sqrt{2/\pi}$, \eqref{H_T IN THE CROSS TERM UPPER BOUND}, \Cref{LOWER BOUND OF TWO DISTANCES} and  $\text{dist}^2(\{\mX_i\},\{ \mX_i^\star\})\leq \frac{\underline{\sigma}^2(\calX^\star) ((1-2p_s)\sqrt{2/\pi} -  \delta_{(N+1)\overline{r}})^2 }{18(2N^2 - 2N +1)(N+1+\sum_{i=2}^{N-1}r_i)(\sqrt{2/\pi} + \delta_{(N+1)\overline{r}})^2}$.

\end{proof}

\section{Proof of \Cref{Local Convergence of SGD in the l1 sensing_Theorem}}
\label{Local Convergence Proof of SGD l1 Tensor Sensing}

\begin{proof}
To utilize the robust regularity condition of \Cref{Robust Regularity condition of factor tensor} in the derivation of \Cref{Local Convergence of SGD in the l1 sensing_Theorem}, we need to prove conditions in \Cref{Robust Regularity condition of factor tensor}.
Due to the retraction operation, we can guarantee that $L(\mX_i^{(t)})$ are orthonormal. In addition, we assume that
\begin{eqnarray}
    \label{upper bound of error distance any t}
    \text{dist}^2(\{\mX_i^{(t)}\},\{ \mX_i^\star\})\leq \frac{\underline{\sigma}^2(\calX^\star) ((1-2p_s)\sqrt{2/\pi} -  \delta_{(N+1)\overline{r}})^2 }{18(2N^2 - 2N +1)(N+1+\sum_{i=2}^{N-1}r_i)(\sqrt{2/\pi} + \delta_{(N+1)\overline{r}})^2},
\end{eqnarray}
which can be proven later, and following \eqref{upper bound TT spctral norm robust}, then obtain
\begin{eqnarray}
    \label{upper bound of LN t}
    \sigma_1^2({\calX^{(t)}}^{\<i \>}) \leq \frac{9\ol{\sigma}^2(\calX^\star)}{4}, i\in[N-1].
\end{eqnarray}

Next, we define the best rotation matrices as following:
\begin{eqnarray}
    \label{the definition of orthonormal matrix R}
    (\mR_1^{(t)},\dots,\mR_{N-1}^{(t)}) = \argmin_{\mR_i\in\O^{r_i\times r_i}, \atop i\in [N-1]}\sum_{i=1}^{N-1} \ol{\sigma}^2(\calX^\star)\|L({\mX}_i^{(t)})-L_{\mR}({\mX}_i^\star)\|_F^2 + \|L(\mX_N^{(t)})-L_{\mR}({\mX}_N^\star)\|_2^2.
\end{eqnarray}

Now we can prove the assumption $\text{dist}^2(\{\mX_i^{(t+1)}\},\{ \mX_i^\star\})\leq \text{dist}^2(\{\mX_i^{(t)}\},\{ \mX_i^\star\})$, expand $\text{dist}^2(\{\mX_i^{(t+1)}\},\{ \mX_i^\star\})$ and subsequently derive
\begin{eqnarray}
    \label{expansion of distance in tensor factorization-L1 loss}
    &\!\!\!\!\!\!\!\!&\text{dist}^2(\{\mX_i^{(t+1)}\},\{ \mX_i^\star\})\nonumber\\
    &\!\!\!\!=\!\!\!\!& \sum_{i=1}^{N-1} \ol{\sigma}^2(\calX^\star)\| L(\mX_i^{(t+1)})-L_{\mR^{(t+1)}}(\mX_i^\star) \|_F^2+\|L(\mX_N^{(t+1)})-L_{\mR^{(t+1)}}(\mX_N^\star)\|_2^2  \nonumber\\
    &\!\!\!\!\leq\!\!\!\!&\sum_{i=1}^{N-1} \ol{\sigma}^2(\calX^\star)\bigg\| L(\mX_i^{(t)})-L_{\mR^{(t)}}(\mX_i^\star) -\frac{\mu_t}{\ol{\sigma}^2(\calX^\star)}\mathcal{P}_{\text{T}_{L({\mX}_i)} \text{St}}\big(\partial_{L({\mX}_{i})}F(\mX_1^{(t)}, \dots, \mX_N^{(t)})\big)\bigg\|_F^2\nonumber\\
    &\!\!\!\!\!\!\!\!&+\|L(\mX_N^{(t)})-L_{\mR^{(t)}}(\mX_N^\star)-\mu_t\partial_{L({\mX}_{N})}F(\mX_1^{(t)}, \dots, \mX_N^{(t)}) \|_2^2\nonumber\\
    &\!\!\!\!=\!\!\!\!&\text{dist}^2(\{\mX_i^{(t)}\},\{ \mX_i^\star\})-2\mu_t\sum_{i=1}^{N} \bigg\< L(\mX_i^{(t)})-L_{\mR^{(t)}}(\mX_i^\star),\mathcal{P}_{\text{T}_{L({\mX}_i)} \text{St}}\big(\partial_{L({\mX}_{i})}F(\mX_1^{(t)}, \dots, \mX_N^{(t)})\big)\bigg\>\nonumber\\
    &\!\!\!\!\!\!\!\!&+\mu_t^2\bigg(\frac{1}{\ol{\sigma}^2(\calX^\star)}\sum_{i=1}^{N-1}\|\mathcal{P}_{\text{T}_{L({\mX}_i)} \text{St}}\big(\partial_{L({\mX}_{i})}F(\mX_1^{(t)}, \dots, \mX_N^{(t)})\big)\|_F^2+\|\partial_{L({\mX}_{N})}F(\mX_1^{(t)}, \dots, \mX_N^{(t)}) \|_2^2\bigg),
\end{eqnarray}
where the first inequality follows the nonexpansiveness property of \Cref{NONEXPANSIVENESS PROPERTY OF POLAR RETRACTION_1}.

Based on \eqref{PROJECTED GRADIENT DESCENT SQUARED TERM 1 to N robust L1 loss}, we can easily obtain
\begin{eqnarray}
    \label{RIEMANNIAN FACTORIZATION SQUARED TERM UPPER BOUND robust L1 loss}
    &\!\!\!\!\!\!\!\!&\frac{1}{\ol{\sigma}^2(\calX^\star)}\sum_{i=1}^{N-1}\|\mathcal{P}_{\text{T}_{L({\mX}_i)} \text{St}}\big(\partial_{L(\mX_{i})}F(\mX_1^{(t)}, \dots, \mX_N^{(t)})\big)\|_F^2+\|\partial_{L(\mX_{N})}F(\mX_1^{(t)}, \dots, \mX_N^{(t)}) \|_2^2\nonumber\\
    &\!\!\!\!\leq\!\!\!\!&\frac{1}{\ol{\sigma}^2(\calX^\star)}\sum_{i=1}^{N-1}\|\partial_{L(\mX_{i})}F(\mX_1^{(t)}, \dots, \mX_N^{(t)})\|_F^2+\|\partial_{L(\mX_{N})}F(\mX_1^{(t)}, \dots, \mX_N^{(t)}) \|_2^2\nonumber\\
    &\!\!\!\!\leq\!\!\!\!&\frac{9N-5}{4}(\sqrt{2/\pi} +\delta_{2\overline{r}})^2,
\end{eqnarray}
where the first inequality follows from the fact that for any matrix $\mB = \calP_{\text{T}_{L({\mX}_i)} \text{St}}(\mB) + \calP_{\text{T}_{L({\mX}_{i})} \text{St}}^{\perp}(\mB)$ where $\calP_{\text{T}_{L({\mX}_i)} \text{St}}(\mB)$ and $\calP_{\text{T}_{L({\mX}_{i})} \text{St}}^{\perp}(\mB)$ are orthogonal, we have $\|\calP_{\text{T}_{L({\mX}_i)} \text{St}}(\mB)\|_F^2\leq \|\mB\|_F^2$.

Hence, combing the robust regularity condition in \Cref{Robust Regularity condition of factor tensor} and \eqref{RIEMANNIAN FACTORIZATION SQUARED TERM UPPER BOUND robust L1 loss}, we have
\begin{eqnarray}
    \label{Conclusion of SGD in the L1 sensing}
    \text{dist}^2(\{\mX_i^{(t+1)}\},\{ \mX_i^\star\})&\!\!\!\!\leq\!\!\!\!&\text{dist}^2(\{\mX_i^{(t)}\},\{ \mX_i^\star\}) - \frac{\underline{\sigma}(\calX^\star)((1-2p_s)\sqrt{2/\pi} - \delta_{(N+1)\overline{r}})}{2\sqrt{2(N+1+\sum_{i=2}^{N-1}r_i)}\ol{\sigma}(\calX^\star)}\mu_t\text{dist}(\{\mX_i^{(t)}\},\{ \mX_i^\star\}) \nonumber\\
    &\!\!\!\!\!\!\!\!&+\frac{9N-5}{4}(\sqrt{2/\pi}+\delta_{(N+1)\overline{r}})^2\mu_t^2.
\end{eqnarray}

Following the same analysis of \eqref{The expansion of difference two tensors final conclusion2} in {Appendix} \ref{Proof of the induction in the full tensor}, we can set $\mu_t=\lambda q^t$ where we respectively select $$\lambda=\frac{(1-2p_s)\sqrt{2/\pi} - \delta_{(N+1)\overline{r}}}{\sqrt{2(N+1+\sum_{i=2}^{N-1}r_i)}(9N - 5)(\sqrt{2/\pi} + \delta_{(N+1)\overline{r}})^2\kappa(\calX^\star)}\text{dist}(\{\mX_i^{(0)}\},\{ \mX_i^\star\})$$ and $$q =  \sqrt{1- \frac{((1-2p_s)\sqrt{2/\pi} -  \delta_{(N+1)\overline{r}})^2 }{8(N+1+\sum_{i=2}^{N-1}r_i)(9N - 5)(\sqrt{2/\pi} + \delta_{(N+1)\overline{r}})^2\kappa^2(\calX^\star)}},$$
and then guarantee that
\begin{eqnarray}
    \label{Conclusion of SGD in the L1 sensing2}
    \text{dist}^2(\{\mX_i^{(t)}\},\{ \mX_i^\star\})\leq \text{dist}^2(\{\mX_i^{(0)}\},\{ \mX_i^\star\}) q^{2t}.
\end{eqnarray}

\paragraph{Proof of \eqref{upper bound of error distance any t}} We now prove \eqref{upper bound of error distance any t} by induction. First note that \eqref{upper bound of error distance any t} holds for $t = 0$ which can be proved by combing $\calX^{(0)}\in\calC(b)$ and \Cref{LOWER BOUND OF TWO DISTANCES}. We now assume it holds at $t = t'$, which implies that $\sigma_1^2({\calX^{(t')}}^{\<i \>}) \leq \frac{9\ol{\sigma}^2(\calX^\star)}{4}, i\in[N-1]$. By invoking \eqref{Conclusion of SGD in the L1 sensing2}, we have $\text{dist}^2(\{\mX_i^{(t'+1)}\},\{\mX_i^\star\}) \le \text{dist}^2(\{\mX_i^{(t)}\},\{ \mX_i^\star\})$. Consequently, \eqref{upper bound of error distance any t} also holds  at $ t= t'+1$. By induction, we can conclude that \eqref{upper bound of error distance any t} holds for all $t\ge 0$. This completes the proof.

\end{proof}

\section{Proof of \Cref{Error analysis of truncated spectral initialization}}
\label{proof of truncated spectral initialization}

Before analyzing the truncated spectral initialization, we first define the following restricted Frobenius norm for any tensor $\calH\in\R^{d_1\times\cdots\times d_N}$:
\begin{eqnarray}
    \label{Definition of the restricted F norm}
    \|\calH\|_{F,\ol r} = \hspace{-0.5cm} \max_{\calX\in\R^{d_1\times\cdots\times d_N}, \|\calX\|_F\leq 1, \atop {\rm rank}(\calX)=(r_1,\dots,r_{N-1}) } \hspace{-0.2cm} \<\calH,  \calX \>,
\end{eqnarray}
where $\rank(\calX)$ denotes the TT ranks of $\calX$.

Following the same analysis of \cite[eq. (89)]{qin2024guaranteed}, we have
\begin{eqnarray}
    \label{Upper bound of the truncated spectral initialization}
    \|\calX^{(0)} - \calX^\star \|_F &\!\!\!\!\leq\!\!\!\!& (1 + \sqrt{N-1}) \bigg\|\frac{1}{(1-p_s)m}\sum_{k=1}^my_k\calA_k \setI_{\{ |y_k|\leq |\vy|_{(\lceil p_s m \rceil)}  \}}-\calX^\star\bigg\|_{F,2\ol r}\nonumber\\
    &\!\!\!\!=\!\!\!\!& \max_{\calX\in\R^{d_1\times\cdots\times d_N}, \|\calX\|_F\leq 1, \atop {\rm rank}(\calX)=(2r_1,\dots,2r_{N-1})}\frac{1 + \sqrt{N-1}}{(1-p_s)m} \sum_{k\in \calS'} \big\<y_k\calA_k - \calX^\star, \calX  \big\>,
\end{eqnarray}
where $\calS' = \{k:  |y_k|\leq |\vy|_{(\lceil p_s m \rceil)}, k\in[m] \}$ and $|\calS'| = \lceil(1-p_s)m\rceil$. Since $|\vy|_{(\lceil p_s m \rceil)} \leq \max_{k}|\<\calA_k,\calX^\star \>|$ where $\<\calA_k,  \calX^\star \> \sim\calN(0,\|\calX^\star\|_F^2)$, we first have
\begin{eqnarray}
    \label{concentration inequality of gaussian random variable}
    \P{|\<\calA_k,  \calX^\star \>  | \geq t } \leq 2 e^{-\frac{t^2}{2\|\calX^\star\|_F^2}},
\end{eqnarray}
and it follows that
\begin{eqnarray}
    \label{concentration inequality of max gaussian random variable}
    \P{\max_k|\<\calA_k,  \calX^\star \>  |\leq t } &\!\!\!\!\geq\!\!\!\!&  1 - 2\lceil (1 - p_s) m \rceil e^{-\frac{t^2}{2\|\calX^\star\|_F^2}}\nonumber\\
    &\!\!\!\!\geq \!\!\!\!& 1 -  e^{1-\frac{t^2}{2\|\calX^\star\|_F^2}+ \log ((1 - p_s) m)}.
\end{eqnarray}
Then taking $t = c_1\|\calX^\star\|_F \log ((1 - p_s) m) $ with a positive constant $c_1$, with probability $1 - e^{-\Omega(\log ((1 - p_s) m))}$ we can obtain
\begin{eqnarray}
    \label{upper bound of  max gaussian random variable}
    |\vy|_{(\lceil p_s m \rceil)} \leq \max_k|\<\calA_k,  \calX^\star \>  \leq O(\|\calX^\star\|_F \log ((1 - p_s) m)).
\end{eqnarray}

Next, according to \cite[Appendix E]{qin2024guaranteed}, we can construct an $\epsilon$-net $\{\calX^{(1)}, \dots, \calX^{(p)}\}$ with covering number
\begin{eqnarray}
    \label{covering number of TT structure}
p \leq (\frac{4+\epsilon}{\epsilon})^{4N\ol d\ol r^2 }
\end{eqnarray}
for any TT format tensors $\calX$ with TT ranks $(2r_1,\dots, 2r_{N-1})$ such that
\begin{eqnarray}
    \label{relationship between fixed tensor and any tensor}
    \max_{\calX\in\R^{d_1\times\cdots\times d_N}, \|\calX\|_F\leq 1, \atop {\rm rank}(\calX)=(2r_1,\dots,2r_{N-1})}\sum_{k\in \calS'} \big\<y_k\calA_k - \calX^\star, \calX  \big\> \leq 2 \sum_{k\in \calS'} \big\<y_k\calA_k - \calX^\star, \calX^{(p)}  \big\>
\end{eqnarray}
using $\epsilon = \frac{1}{2N}$.

Note that $\E_{\calA_k}\<y_k\calA_k - \calX^\star, \calX^{(p)}  \> = \E_{\calA_k} (\< \calA_k, \calX^\star  \> +s_k)\<\calA_k, \calX^{(p)}  \> - \<\calX^\star, \calX^{(p)}  \> = \<\calX^\star, \calX^{(p)} \>  - \<\calX^\star, \calX^{(p)} \> = 0$ since each element in $\calA_k$ follows the normal distribution. In addition, $\<y_k\calA_k - \calX^\star, \calX^{(p)}  \>$ is a subgaussian random variable with subgaussian norm $\| \<y_k\calA_k - \calX^\star, \calX^{(p)}  \> \|_{\psi_2}\leq |y_k| \|\<\calA_k,  \calX^{(p)} \>\|_{\psi_2} + \|\calX^\star\|_F \|\calX^{(p)}\|_F\leq O(\log ((1 - p_s) m) \|\calX^\star\|_F)$ where we use $|y_k|\leq |\vy|_{(\lceil p_s m \rceil)}\leq O(\|\log ((1 - p_s) m) \|\calX^\star\|_F)$, $\<\calA_k,  \calX^{(p)} \> \sim\calN(0,\|\calX^{(p)}\|_F^2)$ and $\|\calX^{(p)}\|_F\leq 1$. According to the General Hoeffding’s inequality \cite[Theorem 2.6.2]{vershynin2018high}, we have
\begin{eqnarray}
    \label{concentration inequality of fixed tensor initialization}
    \P{ |\sum_{k\in \calS'} \big\<y_k\calA_k - \calX^\star, \calX^{(p)}  \big\>| \geq t} \leq 2e^{-\frac{c_2t^2}{(1-p_s)m(\log ((1 - p_s) m))^2 \|\calX^\star\|_F^2}},
\end{eqnarray}
where $c_2$ is a positive constant. Combing \eqref{relationship between fixed tensor and any tensor} with \eqref{concentration inequality of fixed tensor initialization}, we further derive
\begin{eqnarray}
    \label{concentration inequality of any tensor initialization}
    &\!\!\!\!\!\!\!\!&\P{ \max_{\calX\in\R^{d_1\times\cdots\times d_N}, \|\calX\|_F\leq 1, \atop {\rm rank}(\calX)=(2r_1,\dots,2r_{N-1})}\sum_{k\in \calS'} \big\<y_k\calA_k - \calX^\star, \calX  \big\> \geq t}\nonumber\\
    &\!\!\!\!\leq \!\!\!\!&\P{ \sum_{k\in \calS'} \big\<y_k\calA_k - \calX^\star, \calX^{(p)}  \big\> \geq \frac{t}{2}}\nonumber\\
    &\!\!\!\!\leq \!\!\!\!&\P{ |\sum_{k\in \calS'} \big\<y_k\calA_k - \calX^\star, \calX^{(p)}  \big\>| \geq \frac{t}{2}}\nonumber\\
    &\!\!\!\!\leq \!\!\!\!& (\frac{4+\epsilon}{\epsilon})^{4N\ol d\ol r^2 }e^{1 -\frac{c_2t^2}{4(1-p_s)m(\log ((1 - p_s) m))^2 \|\calX^\star\|_F^2}}\nonumber\\
    &\!\!\!\!\leq \!\!\!\!& e^{1 -\frac{c_2t^2}{4(1-p_s)m(\log ((1 - p_s) m))^2 \|\calX^\star\|_F^2} + c_3 N \ol d\ol r^2 \log N},
\end{eqnarray}
where $c_3$ is a constant and based on the assumption in \eqref{relationship between fixed tensor and any tensor}, $\frac{4+\epsilon}{\epsilon}=\frac{4+\frac{1}{2N}}{\frac{1}{2N}}=8N+1$.

Taking $t = \log ((1 - p_s) m) \|\calX^\star\|_F\sqrt{(1-p_s)m N \ol d\ol r^2 \log N  }$, with probability $1 - e ^{-\Omega(N \ol d\ol r^2 \log N)} - e^{-\Omega(\log ((1 - p_s) m))}$, we have
\begin{eqnarray}
    \label{upper bound of truncated spectral initialization}
    \|\calX^{(0)} - \calX^\star \|_F \leq O\bigg(\frac{N\ol r \log ((1 - p_s) m) \|\calX^\star\|_F\sqrt{ \ol d \log N } }{\sqrt{(1-p_s)m}}\bigg).
\end{eqnarray}

%

\end{document}